\documentclass[12pt,reqno]{amsart}
\usepackage{amsmath}
\usepackage{amssymb}
\usepackage{verbatim}
\usepackage{mathrsfs}
\usepackage{graphicx}
\usepackage[font=footnotesize]{caption}
\usepackage{subcaption}
\usepackage{tikz}
\usepackage{pgfplots}
\usepackage[capitalise]{cleveref}
\usepackage[noadjust]{cite}

\usepackage[top=0.5in,bottom=0.5in,left=0.7in,right=0.7in]{geometry}

\newcommand{\C}{\mathbb{C}}    %complex number field
\newcommand{\N}{\mathbb{N}}    %natural numbers
\newcommand{\NN}{\mathbb{N}_0} %Natural Nonnegative numbers
\newcommand{\R}{\mathbb{R}}    %real number field
\newcommand{\Z}{\mathbb{Z}}    %integers
\newcommand{\sD}{\mathsf{D}}
\newcommand{\CH}[1]{\mathscr{C}^{#1}(\R)} \newcommand{\Lp}[1]{L_{#1}(\mathbb{R})}
\newcommand{\cd}{\mathcal{R}}  %cascade/refinement operator C: use C or R.
\newcommand{\sm}{\operatorname{sm}}  %smoothness exponent
\newcommand{\sr}{\operatorname{sr}}  %sum rules

\newcommand{\bp}{ \begin{proof} }
\newcommand{\ep}{\hfill \end{proof} }
\newcommand{\be}{ \begin{equation} }
\newcommand{\ee}{ \end{equation} }
\newcommand{\imply}{ \Longrightarrow }
\newcommand{\tp}{\mathsf{T}}  %transpose
\newcommand{\fs}{\operatorname{fsupp}}
\newcommand{\mspan}{\operatorname{span}}
\newcommand{\ld}{\operatorname{ld}}
\newcommand{\pp}{\mathsf{p}}
\newcommand{\pq}{\mathsf{q}}
\newcommand{\lp}[1]{l_{#1}(\mathbb{Z})}
\newcommand{\lrs}[3]{(l_{#1}(\mathbb{Z}))^{#2\times #3}}

\newcommand{\PL}{\Pi}   %ring of all polynomials
\newcommand{\PR}{\mathscr{P}}   %use for polynomial rings
\newcommand{\PV}{\mathscr{V}}
\newcommand{\PB}{\mathscr{B}}

\newcommand{\sd}{\mathcal{S}}  %subdivision operator S
\newcommand{\tz}{\mathcal{T}}  %transition operator T

\newcommand{\err}{E}
\newcommand{\wh}{\widehat}
\renewcommand{\le}{\leqslant}
\renewcommand{\ge}{\geqslant}
\newcommand{\bs}{\backslash}

\newcommand{\bo}{\mathscr{O}} %standard big O notation
\newcommand{\setsp}{\;:\;}     %set separator

\newcommand{\vgu}{\upsilon} %matching filter for matrix filters
\newcommand{\td}{\boldsymbol{\delta}}
\newcommand{\SL}[2]{{W^{#1}_{#2}(\mathbb{R})}}    %1D Sobolev
\newcommand{\ind}{\Lambda} %\triangle

\newtheorem{lemma}{Lemma}
\newtheorem{prop}[lemma]{Proposition}

\newtheorem{theorem}[lemma]{Theorem}

\newtheorem{example}{Example}

\newtheorem{definition}{Definition}

\numberwithin{equation}{section}
\begin{document}

\title{Vector Subdivision Schemes for Arbitrary Matrix Masks}

\author{Bin Han}

\address{Department of Mathematical and Statistical Sciences,
University of Alberta, Edmonton,\quad Alberta, Canada T6G 2G1.
\quad {\tt bhan@ualberta.ca}\quad {\tt http://www.ualberta.ca/$\sim$bhan}
}

\thanks{Research was supported in part by the Natural Sciences and Engineering Research Council of Canada (NSERC) under grant RGPIN-2019-04276.
}

%\thanks{Contact information of corresponding author Bin Han: E-mail: bhan@ualberta.ca, Phone: 1-587-8828375, Fax: 1-780-4926826,  Web: http://www.ualberta.ca/$\sim$bhan}

\makeatletter \@addtoreset{equation}{section} \makeatother

\begin{abstract}
Employing a matrix mask, a vector subdivision scheme is a fast iterative averaging algorithm to compute refinable vector functions for wavelet methods in numerical PDEs and to produce smooth curves in CAGD. In sharp contrast to the well-studied scalar subdivision schemes, vector subdivision schemes are much less well understood, e.g., Lagrange and (generalized) Hermite subdivision schemes are the only studied vector subdivision schemes in the literature. Because many wavelets used in numerical PDEs are derived from refinable vector functions whose matrix masks are not from Hermite subdivision schemes, it is necessary to introduce and study vector subdivision schemes for any general matrix masks in order to compute wavelets and refinable vector functions efficiently. For a general matrix mask, we show that there is only one meaningful way of defining a vector subdivision scheme.
Motivated by vector cascade algorithms and recent study on Hermite subdivision schemes,
we shall define a vector subdivision scheme for any arbitrary matrix mask and then we prove that the convergence of the newly defined vector subdivision scheme is equivalent to the convergence of its associated vector cascade algorithm. We also study convergence rates of vector subdivision schemes.
The results of this paper not only bridge the gaps and establish intrinsic links between vector subdivision schemes and vector cascade algorithms but also strengthen and generalize current known results on Lagrange and (generalized) Hermite subdivision schemes. Several examples are provided to illustrate the results in this paper on various types of vector subdivision schemes with convergence rates.
\end{abstract}


\maketitle

\pagenumbering{arabic}

\section{Motivations and Main Results}
\label{sec:intro}

A subdivision scheme is a fast averaging algorithm
to numerically compute wavelets and their derivatives and integrals for wavelet methods in numerical PDEs, to generate smooth curves in CAGD and isogeometric analysis, and to reconstruct signals in a fast wavelet transform.
Due to their many desired properties and intrinsic connections to splines and wavelets,
scalar subdivision schemes have been extensively studied and well understood in the literature, for example, see \cite{cdm91,dl02,hanbook,hj98,hj06} and many references therein.
Employing an $r\times r$ matrix mask $a=\{a(k)\}_{k\in \Z}: \Z\rightarrow \C^{r\times r}$ instead of a scalar mask (i.e., $r=1$) and following the same fashion as in a scalar subdivision scheme,
a vector subdivision scheme is an iterative averaging algorithm to produce a sequence of vector refinement subdivision data.
However, in sharp contrast to the well-studied scalar subdivision schemes, vector subdivision schemes are more challenging to analyze and much less understood in the literature.

Let us first recall some notations and definitions. Let $r,s\in \N$ be positive integers.
By $\lrs{}{s}{r}$ we denote the linear space of all sequences $u=\{u(k)\}_{k\in \Z}: \Z \rightarrow \C^{s\times r}$.
Similarly,
$\lrs{0}{s}{r}$ consists of all finitely supported sequences $u\in \lrs{}{s}{r}$ such that $\{k\in \Z \setsp u(k)\ne 0\}$ is finite.
Let $a=\{a(k)\}_{k\in \Z}\in \lrs{0}{r}{r}$, which is often called a matrix \emph{mask} in CAGD and a matrix \emph{filter} in multiwavelet theory. A vector subdivision scheme can be conveniently expressed through \emph{the vector/matrix subdivision operator} $\sd_a: \lrs{}{s}{r} \rightarrow \lrs{}{s}{r}$ which is defined to be
\be \label{sd}
(\sd_a v)(j):=2\sum_{k\in \Z} v(k) a(j-2k),\qquad j\in \Z
\ee
for $v=\{v(k)\}_{k\in \Z}\in \lrs{}{s}{r}$. In applications, one often considers $s=1$ for vector-valued data and uses real-valued masks $a$.
A vector subdivision operator in \eqref{sd} can be efficiently implemented through convolution. For $u\in \lrs{0}{r}{s}$ and $v\in \lrs{}{s}{t}$, the convolution $u*v$ is defined to be $[u*v](j):=\sum_{k\in \Z} u(j-k)v(k)$ for all $j\in \Z$.
For $\gamma\in \Z$, the $\gamma$-coset of a mask $a\in \lrs{0}{r}{r}$ is
\be \label{coset}
a^{[\gamma]}(k):=a(\gamma+2k),\qquad k, \gamma\in \Z.
\ee
Then it is straightforward to observe that \eqref{sd} is equivalent to
\[
(\sd_a v)^{[0]}=(\sd_a v)(2\cdot)=
2v*a^{[0]}
\quad \mbox{and} \quad
(\sd_a v)^{[1]}=(\sd_a v)(1+2\cdot)=
2v*a^{[1]}.
\]
Starting with an initial row sequence
$w_0: \Z \rightarrow \C^{1\times r}$, we apply the vector subdivision operator $\sd_a$ on $w_0$ iteratively to produce a sequence of refinement vector-valued data $\{\sd_a^n w_0\}_{n=1}^\infty$. The main goal of a convergent vector subdivision scheme is to guarantee that the refinement data $\{\sd_a^n w_0\}_{n=1}^\infty$, after rescaling if necessary, converge to some smooth nontrivial vector function in a meaningful way.
The most extensively studied and well-known vector subdivision schemes in the literature are
Lagrange and Hermite ones.
For any input vector sequence $w_0\in \lrs{}{1}{r}$,
the (vector-valued) Lagrange subdivision scheme seeks a continuous function $\eta$ such that for every constant $K>0$,
\be \label{lsd:converg}
\lim_{n\to \infty} \sup_{k\in \Z \cap [-2^n K, 2^n K]} |(\sd_a^n w_0)(k) e_\ell-\eta (2^{-n}k)|=0,\qquad \ell=1,\ldots,r,
\ee
where $e_\ell$ is the $\ell$-th unit coordinate column vector in $\R^r$.
See \cite{dm06,hyx05,ms98} and references therein on Lagrange subdivision schemes, which are however called vector subdivision schemes there.
A Hermite subdivision scheme is quite different by using consecutive derivatives. For any input vector sequence $w_0\in \lrs{}{1}{r}$,
the Hermite subdivision scheme of order $r$ seeks a smooth function $\eta\in \CH{r-1}$ such that for every constant $K>0$ and $\ell=1,\ldots,r$,
\be \label{hsd:converg}
\lim_{n\to \infty} \sup_{k\in \Z \cap [-2^n K, 2^n K]} |(\sd_a^n w_0)(k) \sD^{-n}e_\ell-\eta^{(\ell-1)}(2^{-n}k)|=0
\quad \mbox{with}\quad
\sD:=\mbox{diag}(1,2^{-1},\ldots, 2^{1-r}),
\ee
where $\eta^{(\ell-1)}$ stands for the $(\ell-1)$-th order derivative of $\eta$ and the diagonal matrix $\sD$ rescales the original refinement data $\{\sd_a^n w_0\}_{n=1}^\infty$ due to derivatives.
An interpolatory Hermite subdivision scheme employs a special matrix mask $a$ satisfying $a(0)=2^{-1}\sD$
and $a(2k)=0$ for all $k\in \Z\bs\{0\}$.
Interpolatory Hermite subdivision schemes of order $2$ were initially studied in Merrien \cite{mer92} and Dyn and Levin in \cite{dl95}, and further investigated in
\cite{han01,han03,hanbook,zhou00} and references therein.
More recently, Hermite subdivision schemes have been extensively studied by many researchers, e.g.,
see \cite{ccmm21,ch19,cmr14,dm09,han20,hyx05,ms17,ms19,mhc20,rv20} and many references therein.
A Hermite subdivision scheme can be easily restated as a nonstationary Lagrange subdivision scheme, because $w_n:=(\sd_a^n w_0)\sD^{-n}$ satisfies
$w_n=\sd_{\sD^{n-1} a\sD^{-n}} w_{n-1}$ for all $n\in \N$ with the level-dependent masks $\sD^{n-1} a\sD^{-n}$.
Because a Hermite subdivision scheme in \eqref{hsd:converg} involves derivatives for $r>1$,
a Hermite subdivision scheme is often regarded as more difficult to study than a Lagrange one. To unify and generalize Lagrange and Hermite subdivision schemes, a generalized Hermite subdivision scheme of type $\ind$ has been recently introduced and analyzed in \cite{han21}, where
$\ind=\{\nu_1,\ldots,\nu_r\}$ is an ordered multiset of $\{0,\ldots,m\}$, i.e., all $\nu_\ell\in \{0,\ldots,m\}$ are not necessarily distinct for  $\ell=1,\ldots,r$.
A generalized Hermite subdivision scheme of type $\ind$ seeks a function
$\eta\in \CH{m}$ such that for every constant $K>0$ and $\ell=1,\ldots,r$,
\be \label{ghsd:converg}
%\begin{split}
\lim_{n\to \infty} \sup_{k\in \Z \cap [-2^n K, 2^n K]} |(\sd_a^n w_0) \sD_\ind^{-n}e_\ell-\eta^{(\nu_\ell)}(2^{-n}k)|=0
\quad
\mbox{with}\quad
\sD_\ind:=\mbox{diag}(2^{-\nu_1},\ldots, 2^{-\nu_r}).
%\end{split}
\ee
Obviously, a Lagrange or Hermite subdivision scheme is just a generalized Hermite subdivision scheme of type $\ind=\{0,\ldots,0\}$ or $\ind=\{0,1,\ldots,r-1\}$, respectively.
See \cite{han21} for more details on multivariate generalized Hermite subdivision schemes of type $\ind$.
%with or without the interpolation properties.
But for a lot of matrix masks in multiwavelet theory, the Lagrange, Hermite or generalized Hermite subdivision schemes are often divergent, i.e., \eqref{lsd:converg}, \eqref{hsd:converg} and \eqref{ghsd:converg} often fail. In particular, the matrix masks of many non-spline multiwavelets in numerical PDEs cannot be used in any known subdivision schemes to compute derivatives for wavelet collocation methods and integrals for wavelet Galerkin methods in numerical PDEs. These motivate us to study how to properly define a vector subdivision scheme for a general matrix mask.

To explain our motivations for properly defining a vector subdivision scheme,
let us first explain our motivation based on a simple observation.
Let us consider $\{v_n\}_{n=1}^\infty$ such that each $v_n: \Z \rightarrow \C$ is a chosen entry of the refinement vector-valued data $(\sd_a^n w_0) \sD_\ind^{-n}$ in \eqref{ghsd:converg}. More generally, we can simply allow $v_n: \Z \rightarrow \C$ to be a general sequence which is not necessarily generated by any subdivision scheme. We hope that $\{v_n\}_{n=1}^\infty$ converges (in a meaningful way) to some smooth nontrivial function. More precisely, for a given real number $\tau\in \R$, we are interested in whether there exists a continuous function $\eta_\tau$ such that for every constant $K>0$,
\be \label{vsd:tau}
\lim_{n\to \infty}
\sup_{k\in \Z\cap [-2^n K, 2^nK]}
|v_n(k)2^{\tau n} -\eta_\tau(2^{-n}k)|=0.
\ee
It is straightforward to observe from \eqref{vsd:tau} that there exists a unique $\tau_*\in \R\cup\{\pm \infty\}$ given by
\[
\tau_*:=\sup\Big\{ \tau\in \R \setsp \lim_{n\to \infty}
\sup_{k\in \Z\cap [-2^n K, 2^nK]}
|v_n(k)|2^{\tau n}=0,\; \mbox{ for every }\; K>0\Big \}
\]
(we define $\tau_*:=-\infty$ if the above set on the right-hand side is empty)
such that
\begin{enumerate}
\item[(i)] for all $\tau<\tau_*$, the limit in \eqref{vsd:tau} converges to the trivial function $\eta_\tau=0$. Therefore, all the information in $\{v_n\}_{n=1}^\infty$ is lost and the limiting process in \eqref{vsd:tau} is degenerate and trivial;
\item[(ii)] for all $\tau>\tau_*$, the limit in \eqref{vsd:tau} is divergent and there does not exist a continuous function $\eta_\tau$ satisfying \eqref{vsd:tau}. Therefore, the limiting process in \eqref{vsd:tau} is meaningless;
\item[(iii)] for $\tau=\tau_*$, the existence of a limit function $\eta_{\tau_*}$ in \eqref{vsd:tau} is not clear and needs further study.
\end{enumerate}
The above simple observation shows that for the convergence of $\{v_n\}_{n=1}^\infty$ to a nontrivial continuous function in \eqref{vsd:tau}, the original data sequence $\{v_n\}_{n=1}^\infty$ should be properly rescaled by a unique scaling exponent $\tau_*$, which is uniquely determined by the data $\{v_n\}_{n=1}^\infty$.
Hence, we actually don't have any choice or freedom to define a vector subdivision process and the key is to find such a critical exponent $\tau_*$ for every entry of the vector refinement data $\{w_n:=\sd_a^n w_0\}_{n=1}^\infty$ generated by a vector subdivision scheme. If $\{v_n\}_{n=1}^\infty$ comes from the refinement data $\{w_n\}_{n=1}^\infty$, then it is possible that the above unique critical exponent $\tau_*$ may depend on both the mask $a$ and the initial input data $w_0\in \lrs{0}{1}{r}$. This will make things quite complicated for defining vector subdivision schemes. Fortunately, we shall show in this paper that this never happens and $\tau_*$ only depends on the mask $a$.
The above simple observation motivates us to define a general vector subdivision scheme for a general matrix mask.

Our second motivation is from the perspective of refinable vector functions
in multiwavelet theory.
For any convergent (such as Lagrange, Hermite or generalized Hermite) subdivision scheme,
because a subdivision scheme is iterative in nature,  its basis vector function $\phi=[\phi_1,\ldots,\phi_r]^\tp$ should naturally be a refinable vector function, that is, $\phi$ satisfies the following vector refinement equation:
\be \label{refeq}
\phi=2\sum_{k\in \Z} a(k) \phi(2\cdot-k),
\quad \mbox{or equivalently,}\quad
\wh{\phi}(2\xi)=\wh{a}(\xi)\wh{\phi}(\xi),
\ee
where the Fourier transform used in this paper is defined to be $\wh{f}(\xi):=\int_\R f(x) e^{-ix\xi} dx$ for $\xi\in \R$ and can be naturally extended to tempered distributions via duality.
Here $\wh{a}$ in \eqref{refeq} is the \emph{symbol} or \emph{Fourier series} of $a\in \lrs{0}{r}{r}$ defined by
\be \label{fourier}
\wh{a}(\xi):=\sum_{k\in \Z} a(k) e^{-ik\xi},\qquad \xi\in \R,
\ee
which is an $r\times r$ matrix of $2\pi$-periodic trigonometric polynomials. Note that $\wh{u*v}(\xi)=\wh{u}(\xi)\wh{v}(\xi)$.
In multiwavelet theory, there are many smooth compactly supported refinable vector functions in $(\CH{m})^r$ (e.g., see \cite{cjr02,han01,han03,han09,hanbook,jj03,jrz98}) satisfying \eqref{refeq} with finitely supported matrix masks $a\in \lrs{0}{r}{r}$. Therefore, it is an important topic for properly defining vector subdivision schemes to efficiently compute the refinable vector functions and their derivatives on a fine grid $2^{-n}\Z$ for some large $n\in \N$. This task is particularly crucial for wavelet collocation methods to solve nonlinear PDEs and for wavelet Galerkin methods in numerical PDEs.
For scalar refinable functions (i.e., $r=1$), a scalar subdivision scheme is a major popular tool to achieve this goal (e.g., \cite{cdm91,dl02,hanbook,hj98} and \cite[Theorem~2.1]{hj06}). However, for the vector case $r>1$, the generalized Hermite subdivision schemes with such masks $a$ are often divergent. Hence, from the perspective of refinable vector functions, it is necessary to introduce the definition of a vector subdivision scheme for a general matrix mask.

To state our definition of a vector subdivision scheme for a general matrix mask, let us recall some necessary notations.
For two smooth functions $\wh{f}$, $\wh{g}$ and for $m\in \NN:=\N\cup\{0\}$, throughout the paper we shall adopt the following convention:
\be \label{bo}
\wh{f}(\xi)=\wh{g}(\xi)+\bo(|\xi|^{m+1}),\quad \xi \to 0 \quad \mbox{just stands for}\quad
\wh{f}^{(j)}(0)=\wh{g}^{(j)}(0)\quad \forall\; j=0,\ldots,m.
\ee
For $m\in \NN$ and a $2\pi$-periodic trigonometric polynomial $\wh{u}$, we define
\be \label{ld}
\ld_m(\wh{u}):=\begin{cases}
j, &\text{if $j\in \{0,\ldots,m\}$, $\wh{u}^{(j)}(0)\ne 0$ and $\wh{u}^{(\ell)}(0)=0$ for all $\ell=0,\ldots, j-1$},\\
m, &\text{if $\wh{u}^{(\ell)}(0)=0$ for all $\ell=0,\ldots,m$}.
\end{cases}
\ee
For $a\in \lrs{0}{r}{r}$ and $m\in \NN$,
we say that
$\vgu_a\in \lrs{0}{1}{r}$ is \emph{an order $m+1$ matching filter of $a$} if
\be \label{mfilter}
\wh{\vgu_a}(0)\ne 0 \quad \mbox{and}\quad
\wh{\vgu_a}(2\xi)\wh{a}(\xi)=\wh{\vgu_a}(\xi)+\bo(|\xi|^{m+1}),\qquad \xi \to 0.
\ee

Now we are ready to define a vector subdivision scheme for a general matrix mask $a\in \lrs{0}{r}{r}$.

\begin{definition}\label{def:vsd}
\normalfont
Let $r\in \N$ and $m\in \NN$. Let $a\in \lrs{0}{r}{r}$ be a finitely supported matrix mask
and $\vgu_a\in \lrs{0}{1}{r}$ be an order $m+1$ matching filter of $a$ satisfying \eqref{mfilter}.
We say that \emph{the vector subdivision scheme with mask $a$ is $\CH{m}$ convergent} if for every input vector sequence $w_0\in \lrs{}{1}{r}$, there exists a $\CH{m}$ function $\eta: \R\rightarrow \C$ such that for every constant $K>0$ and $u\in (\lp{0})^r$,
\be \label{vsd:converg}
\lim_{n\to \infty} \max_{k\in \Z \cap[-2^n K, 2^n K]}
|[(\sd_a^n w_0)*u](k) 2^{j n}- \beta_j \eta^{(j)}(2^{-n}k)|=0
\;\; \mbox{with}\;\;
j:=\ld_m(\wh{\vgu_a}\wh{u}),\; \beta_j:=\tfrac{[\wh{\vgu_a}\wh{u}]^{(j)}(0)}{i^j j!}.
\ee
In fact, the above definition of $j\in \{0,\ldots, m\}$ and $\beta_j\in \C$ in \eqref{vsd:converg} is equivalent to saying that
\be \label{vguau}
\wh{\vgu_a}(\xi)\wh{u}(\xi)=\beta_j (i\xi)^{j}+\bo(|\xi|^{j+1}), \quad \xi \to0 \quad \mbox{such that $\beta_j\ne 0$ if $j<m$}.
\ee
\end{definition}

Before justifying \cref{def:vsd},
we make some remarks about the matching filter $\vgu_a$ in \cref{def:vsd}.
It is important to observe that \eqref{mfilter} for a matching filter $\vgu_a\in \lrs{0}{1}{r}$ only depends on the values $\{\wh{\vgu_a}^{(j)}(0)\}_{j=0}^m$ instead of the filter $\vgu_a$ itself. We shall discuss
matching filters $\vgu_a$ from the perspective of refinable vector functions in \cref{sec:vca}. In particular, if the matrix $\wh{a}(0):=\sum_{k\in \Z} a(k)$ satisfies the condition:
\be \label{cond:a}
\mbox{$1$ is a simple eigenvalue of $\wh{a}(0)$ and all its other eigenvalues are less than $2^{-m}$ in modulus},
\ee
then up to a multiplicative constant,
all the values $\{\wh{\vgu_a}^{(j)}(0)\}_{j=0}^m$ in \eqref{mfilter} are uniquely determined by the following recursive formula (e.g., see \cite[Proposition~3.1]{han03} or
\cite[(5.6.10)]{hanbook}):
\be \label{vgua:value}
\wh{\vgu_a}(0)\wh{a}(0)=\wh{\vgu_a}(0)
\;\; \mbox{and}\;\;
\wh{\vgu_a}^{(j)}(0)=\sum_{k=0}^{j-1} \frac{2^k j!}{k!(j-k)!} \wh{\vgu_a}^{(k)}(0)\wh{a}^{(j-k)}(0)[I_r-2^j\wh{a}(0)]^{-1},\quad 1\le j\le m.
\ee
For the scalar case $r=1$ and $\wh{a}(0):=\sum_{k\in \Z} a(k)=1$, \eqref{cond:a} is trivially true for all $m\in \NN$ and hence, the values
$\{\wh{\vgu_a}^{(j)}(0)\}_{j=0}^m$ satisfying \eqref{mfilter} are uniquely determined by \eqref{vgua:value} under the normalization condition $\wh{\vgu_a}(0)=1$, from which we further conclude that \eqref{vguau} with $r=1$ is satisfied if and only if
\[
\wh{u}(\xi)=\beta_j (i\xi)^{j}+\bo(|\xi|^{j+1}), \quad \xi \to0 \quad \mbox{such that $\beta_j\ne 0$ if $j<m$},
\]
which does not explicitly involve the matching filter $\vgu_a\in \lp{0}$.
By $\td$ we denote the \emph{Dirac sequence} such that $\td(0):=1$ and $\td(k):=0$ for all $k\in \Z\bs\{0\}$. Due to $\sD_\ind=\mbox{diag}(2^{-\nu_1},\ldots, 2^{-\nu_r})$, we observe that $(\sd_a^n w_0)(k)\sD_\ind^{-n} e_\ell$ in \eqref{ghsd:converg} can be equivalently rewritten as
$[(\sd_a^n w_0)*u](k) 2^{ \nu_\ell n}$ with $u:=\td e_\ell$. Therefore, \eqref{vsd:converg} in \cref{def:vsd} obviously covers all the special cases in \eqref{lsd:converg}, \eqref{hsd:converg} and \eqref{ghsd:converg}.
%We shall use $I_r$ for the $r\times r$ identity matrix.

In the following, we shall justify \cref{def:vsd} from the perspectives of refinable vector functions and vector cascade algorithms.
To justify \cref{def:vsd} and to explain the rescaling factor in \eqref{vsd:converg} from the perspective of refinable vector functions, we shall prove the following result in \cref{sec:proof}.

\begin{theorem}\label{thm:mfilter}
Let $r\in \N$, $m\in \NN$ and $u\in (\lp{0})^r$.
Let $a\in \lrs{0}{r}{r}$ be a finitely supported matrix mask and $\vgu_a\in \lrs{0}{1}{r}$ be an order $m+1$ matching filter of the mask $a$. Suppose that there exist a real number $\tau\in [0,\infty)$ and a continuous vector function $\eta$ on $\R$ such that
\be \label{phi:deriv:u}
\lim_{n\to \infty} \| [(\sd_a^n (\td I_r))*u](\cdot)2^{\tau n}- \eta(2^{-n}\cdot)
\|_{(\lp{\infty})^{r}}=0.
\ee
Then the following statements hold:
\begin{enumerate}
\item[(1)] $[\wh{\vgu_a}\wh{u}]^{(\ell)}(0)=0$ for all $\ell\in \{0,\ldots,m\}$ satisfying $\ell<\tau$, the following limit
\be \label{phi:deriv:uc}
\lim_{n\to \infty} \| [(\sd_a^n (\td I_r))*(u*c)](\cdot)2^{\tau n}- \wh{c}(0)\eta(2^{-n}\cdot)
\|_{(\lp{\infty})^{r}}=0,\qquad \forall\; c\in \lp{0}
\ee
holds, and $\eta$ in \eqref{phi:deriv:u} must be a compactly supported refinable vector function satisfying
\be \label{limit:eta}
\wh{\eta}(2\xi)=2^{\tau}\wh{a}(\xi)\wh{\eta}(\xi)
\quad \mbox{and}\quad
\wh{\eta}(\xi)=\lim_{n\to \infty} \Big[2^{\tau n} \Big(\prod_{j=1}^{n} \wh{a}(2^{-j}\xi)\Big) \wh{u}(2^{-n}\xi)\Big],\qquad \xi\in \R.
\ee

\item[(2)] If  $\tau\in [0,m]$ and $\eta$ in \eqref{phi:deriv:u} is not identically zero and if \eqref{cond:a} holds, then $\tau=j:=\ld_m(\wh{\vgu_a}\wh{u})$ must be an integer, $\phi\in (\CH{j})^r$, and $\eta=\beta_j \phi^{(j)}$, where $\beta_j:=\frac{[\wh{\vgu_a}\wh{u}]^{(j)}{(0)}}{i^j j!}\ne 0$ and $\phi$ is the unique vector of compactly supported distributions satisfying $\wh{\phi}(2\xi)=\wh{a}(\xi)\wh{\phi}(\xi)$ and $\wh{\vgu_a}(0)\wh{\phi}(0)=1$.
\item[(3)] If $\tau=j\in \{0,\ldots,m\}$ and $\beta_j:=\frac{[\wh{\vgu_a}\wh{u}]^{(j)}(0)}{i^jj!}\ne 0$, then the vector function $\eta$ in \eqref{phi:deriv:u} cannot be identically zero, $j=\ld_m(\wh{\vgu_a}\wh{u})$ must hold, and there exists a compactly supported vector function $\varphi\in (\CH{j})^r$ such that $\eta=\beta_j \varphi^{(j)}$, $\wh{\varphi}(2\xi)=\wh{a}(\xi)\wh{\varphi}(\xi)$ and $\wh{\vgu_a}(0)\wh{\varphi}(0)=1$.
\end{enumerate}
\end{theorem}

Refinable vector functions are often studied through vector cascade algorithms. Thus, it is not surprising for a vector subdivision scheme to have a close relation with a vector cascade algorithm.
For $1\le p\le \infty$, the \emph{refinement operator}
$\cd_a: (\Lp{p})^r \rightarrow (\Lp{p})^r$
is defined to be
\be \label{cd}
\cd_a f:=2\sum_{k\in \Z} a(k) f(2\cdot-k),\qquad f\in (\Lp{p})^r.
\ee
Note that $\wh{\cd_a f}(2\xi)=\wh{a}(\xi)\wh{f}(\xi)$.
Hence, a refinable vector function $\phi$ satisfying $\wh{\phi}(2\xi)=\wh{a}(\xi)\wh{\phi}(\xi)$ is a fixed point of $\cd_a$.
For $m\in \NN$,
recall that $\|f\|_{\CH{m}}:=\sum_{j=0}^m \|f^{(j)}\|_{\CH{}}<\infty$ for $f\in \CH{m}$.
For $m\in \NN$ and an order $m+1$ matching filter $\vgu_a\in \lrs{0}{1}{r}$ of a matrix mask $a\in \lrs{0}{r}{r}$, we say that
\emph{the vector cascade algorithm with mask $a$ and an order $m+1$ matching filter $\vgu_a$ of the mask $a$ is $\CH{m}$ convergent} if
the cascade sequence $\{\cd_a^n f\}_{n=1}^\infty$ of vector functions is a Cauchy sequence in $(\CH{m})^r$ for every compactly supported vector function $f\in (\CH{m})^r$ satisfying
\be \label{initialf}
\wh{\vgu_a}(0)\wh{f}(0)=1 \quad \mbox{and}\quad
\wh{\vgu_a}(\xi)\wh{f}(\xi+2\pi k)=\bo(|\xi|^{m+1}),\quad \xi \to 0,\; \forall\; k\in \Z\bs\{0\}.
\ee
It is known in \cite[Theorem~4.3]{han03} or \cite[Theorem~5.6.11]{hanbook} that
the vector cascade algorithm with a mask $a\in \lrs{0}{r}{r}$ and an order $m+1$ matching filter $\vgu_a\in \lrs{0}{1}{r}$ of the mask $a$ is $\CH{m}$ convergent if and only if $\sm_\infty(a)>m$, where the quantity $\sm_\infty(a)$ is defined in \cite{han03,hanbook} (see \cref{sec:vca} for details). In addition, if $\phi$ is a compactly supported vector refinable function satisfying $\wh{\phi}(2\xi)=\wh{a}(\xi)\wh{\phi}(\xi)$ with $\wh{\vgu_a}(0)\wh{\phi}(0)=1$, then $\sm_\infty(a)>m$ implies that $\phi\in (\CH{m})^r$, $\lim_{n\to \infty} \|\cd_a^n f- \phi\|_{(\CH{m})^r}=0$ and the condition in \eqref{cond:a} on $\wh{a}(0)$ must hold.
For $m\in \NN$ and $\vgu_a\in \lrs{0}{1}{r}$,
we define
\be \label{Vm}
\PV_{m,\vgu_a}:=\{ u\in (\lp{0})^r \setsp \wh{\vgu_a}(\xi)\wh{u}(\xi)=\bo(|\xi|^{m+1}),
\quad \xi\to 0\}.
\ee
To justify \cref{def:vsd} from the perspective of vector cascade algorithms,
we shall prove the following result in \cref{sec:proof}  establishing the convergence equivalence between vector subdivision schemes and vector cascade algorithms.

\begin{theorem}\label{thm:main}
Let $r\in \N$ and $m\in \NN$. Let $a\in \lrs{0}{r}{r}$ and let $\vgu_a\in \lrs{0}{1}{r}$ be an order $m+1$ matching filter of $a$.
Let $\PB \subseteq \PV_{m-1,\vgu_a}$ such that $\PV_{m,\vgu_a}\subseteq \mspan\{ u(\cdot-k) \setsp u\in \PB, k\in \Z\}$. Suppose
\be \label{cond:a:0}
1 \mbox{ is a simple eigenvalue of } \wh{a}(0) \mbox{ and all } 2^k, k\in \N \mbox{ are not eigenvalues of } \wh{a}(0).
\ee
It is well known (e.g., see \cite[Theorem~5.1.3]{hanbook}) that there is a unique vector $\phi$ of compactly supported distributions satisfying $\wh{\phi}(2\xi)=\wh{a}(\xi)\wh{\phi}(\xi)$ and $\wh{\vgu_a}(0)\wh{\phi}(0)=1$.
Then the following statements are equivalent to each other:
\begin{enumerate}
\item[(1)] $\sm_\infty(a)>m$, or equivalently, the vector cascade algorithm with the mask $a$ and the order $m+1$ matching filter $\vgu_a$ of the mask $a$ is $\CH{m}$ convergent.
\item[(2)] The refinable vector function $\phi$ belongs to $(\CH{m})^r$ and for all $u\in (\lp{0})^r$,
\be \label{phi:deriv}
\lim_{n\to \infty} \| [(\sd_a^n (\td I_r))*u](\cdot)2^{jn}-\beta_j \phi^{(j)}(2^{-n}\cdot)
\|_{(\lp{\infty})^{r}}=0
 \quad \mbox{with}\quad
j:=\ld_m(\wh{\vgu_a}\wh{u}),\; \beta_j:=\tfrac{[\wh{\vgu_a}\wh{u}]^{(j)}(0)}{i^j j!}.
\ee

\item[(3)] The vector subdivision scheme with mask $a$ is $\CH{m}$ convergent in the sense of \cref{def:vsd}.

\item[(4)] For every $w_0\in \lrs{}{1}{r}$, there exists a $\CH{m}$ function $\eta: \R \rightarrow \C$ such that \eqref{vsd:converg} holds for all constants $K>0$ and $u\in \PB$.

\item[(5)] $\phi\in (\CH{m})^r$ (the condition $\phi\in (\CH{m})^r$ is not required in advance if $[\wh{\vgu_a}\wh{u}]^{(m)}(0)=0$ for all $u\in \PB$), and
    for all $u\in \PB$,
\be \label{phi:deriv:B}
\lim_{n\to \infty} \| [(\sd_a^n (\td I_r))*u](\cdot)2^{ mn}-\beta_m \phi^{(m)}(2^{-n}\cdot)
\|_{(\lp{\infty})^{r}}=0
 \quad \mbox{with}\quad \beta_m:=\tfrac{[\wh{\vgu_a}\wh{u}]^{(m)}(0)}{i^m m!},\qquad \forall\; u\in \PB.
\ee
\end{enumerate}
Moreover, all the items (1)--(5) must hold if
\begin{enumerate}
\item[(6)]
the refinable vector function $\phi$ belongs to $(\CH{m})^r$ and
the integer shifts of $\phi$ are stable, i.e.,
\be \label{stability}
\mspan\{\wh{\phi}(\xi+2\pi k) \setsp k\in \Z\}=\C^r,\qquad \forall\; \xi\in \R.
\ee
\end{enumerate}
\end{theorem}

%Item (4) of \cref{thm:main} provides an equivalent definition of a vector subdivision scheme in \cref{def:vsd}.
We shall discuss  in \cref{sec:hsd} how to obtain such a set $\PB\subseteq \PV_{m-1,\vgu_a}$ in \cref{thm:main}.
Item (2) of \cref{thm:main} gives us a way of computing $\phi$ and its derivatives through vector subdivision schemes but without any convergence rate. Fast convergent vector subdivision schemes are important in computational mathematics. For this purpose, we shall prove the following result in \cref{sec:proof}.

\begin{theorem}\label{thm:fastconverg}
Let $a\in \lrs{0}{r}{r}$ such that $m<\sm_\infty(a)\le m+1$ for some $m\in \NN$. Then \eqref{cond:a} must hold and
let $\vgu_a\in \lrs{0}{1}{r}$ be an order $m+1$ matching filter of the matrix mask $a$ as given in \eqref{vgua:value}.
Let $\phi$ be the unique compactly supported refinable vector function satisfying $\wh{\phi}(2\xi)=\wh{a}(\xi)\wh{\phi}(\xi)$ and $\wh{\vgu_a}(0)\wh{\phi}(0)=1$.
Note that $\sm_\infty(a)>m$ implies $\phi\in (\CH{m})^r$.
For $u\in (\lp{0})^r$, define $j:=\ld_m(\wh{\vgu_a}\wh{u})$ and $\beta_j:=\tfrac{[\wh{\vgu_a}\wh{u}]^{(j)}(0)}{i^j j!}$ as in \eqref{phi:deriv}.
If
\be \label{bettermatch}
\wh{\vgu_a}(\xi)\wh{u}(\xi)=\beta_j (i\xi)^j+\bo(|\xi|^{j+s}),\quad \xi \to 0
\quad \mbox{with $s\in \N$ satisfying}\quad
1\le s\le m+1-j,
\ee
then for any $0<\varepsilon<\nu$ with $\nu:=\min(s, \sm_\infty(a)-j)>0$, there is a constant $C>0$ such that
\be \label{phi:fastconverg}
\| [(\sd_a^n (\td I_r))*u](\cdot)2^{jn}-\beta_j \phi^{(j)}(2^{-n}\cdot)
\|_{(\lp{\infty})^{r}}\le C 2^{-(\nu-\varepsilon) n},\qquad \forall\; n\in \N,
\ee
and for every $w_0\in \lrs{}{1}{r}$ and constant $K>0$,
\be \label{vsd:fastconverg}
\max_{k\in \Z\cap[-2^nK,2^nK]}
\| [(\sd_a^n w_0)*u](k)2^{jn}-\beta_j \eta^{(j)}(2^{-n}k)
\|
\le C C_{K,N,w_0} 2^{-(\nu-\varepsilon) n},\qquad \forall\; n\in \N,
\ee
where $\eta:=\sum_{k\in \Z} w_0(k) \phi(\cdot-k)$, $C_{K,N,w_0}:=\sum_{k=-K-N}^{K+N}\|w_0(k)\|<\infty$ and $N\in \N$ is a fixed integer such that both the matrix mask $a$ and the sequence $u$ are supported inside $[-N,N]$.
\end{theorem}

The structure of the paper is as follows.
To further explain our motivations, in \cref{sec:vca} we shall recall the definitions and related results on refinable vector functions and vector cascade algorithms. We shall present the definition of the key quantity $\sm_p(a)$ with $1\le p\le \infty$ in \cref{sec:vca} and prove in \cref{thm:smp} that $\sm_p(a)$ is independent of the choice of a matching filter $\vgu_a$.
Then we shall prove \cref{thm:main,thm:mfilter,thm:fastconverg} in \cref{sec:proof}.
In \cref{sec:hsd}, we shall explain how our results in this paper strengthen and improve known results on Lagrange, Hermite, and generalized Hermite subdivision schemes.
In \cref{sec:ex}, we provide a few examples of various types of vector subdivision schemes to illustrate the results in this paper and to confirm the theoretical convergence rates in \cref{thm:fastconverg} of vector subdivision schemes.

\section{Refinable Vector Functions and Vector Cascade Algorithms}
\label{sec:vca}

In order to further justify \cref{def:vsd} for vector subdivision schemes and to prove our main results in \cref{thm:mfilter,thm:main,thm:fastconverg} on vector subdivision schemes,
in this section we shall recall some related results on refinable vector functions and vector cascade algorithms first. These results also motivate us to define vector subdivision schemes in \cref{def:vsd} from the perspectives of refinable vector functions and vector cascade algorithms.

Because a vector subdivision scheme is iterative in nature, its basis vector function in any meaningful vector subdivision scheme should be a refinable vector function. Therefore, it appears very natural for us to define a vector subdivision scheme according to the properties of such a refinable vector function.
Let us first recall a result from \cite[Proposition~5.6.2]{hanbook} (also see \cite[Propositions~3.1 and 3.2]{han03}) %\cite[Theorem~2.4]{han01}, and \cite{cjr02})
on refinable vector functions. The following result on refinable vector functions is a special case of \cite[Proposition~5.6.2]{hanbook}.

\begin{prop}\label{prop:phi} (\cite[Proposition~5.6.2]{hanbook})
Let $r\in \N$, $m\in \NN$ and $a\in \lrs{0}{r}{r}$. If $\phi=[\phi_1,\ldots,\phi_r]^\tp$ is a compactly supported refinable vector function in $(\CH{m})^r$ satisfying $\wh{\phi}(2\xi)=\wh{a}(\xi)\wh{\phi}(\xi)$
and $\mspan\{\wh{\phi}(2\pi k) \setsp k\in \Z\}=\C^r$, then $\wh{\phi}(0)\ne 0$ and all the following statements hold:
\begin{enumerate}
\item[(1)] $\wh{a}(0):=\sum_{k\in \Z} a(k)$ must satisfy \eqref{cond:a}. Consequently, all the values $\{\wh{\vgu_a}^{(j)}(0)\}_{j=0}^m$ in \eqref{mfilter} are uniquely determined by
    $\wh{\vgu_a}(0)\wh{\phi}(0)=1$ and
    the recursive formula in \eqref{vgua:value}.

\item[(2)] For every polynomial $\pp\in \PL_m$ (i.e., $\pp$ is a polynomial of degree at most $m$), the following two identities hold: $\pp=\sum_{k\in \Z} (\pp*\vgu_a)(k)\phi(\cdot-k)$, where $\pp*\vgu_a:=\sum_{k\in \Z} \pp(\cdot-k) \vgu_a(k)$, and
\be \label{vgu:phi}
\wh{\vgu_a}(\xi)\wh{\phi}(\xi+2\pi k)=\td(k) +\bo(|\xi|^{m+1}),\qquad \xi \to0,\; \forall\; k\in \Z.
\ee
\end{enumerate}
If in addition $\mspan\{\wh{\phi}(\pi+2\pi k) \setsp k\in \Z\}=\C^r$, then
\begin{enumerate}
\item[(3)] the matrix mask $a$ must have order $m+1$ sum rules with a matching filter $\vgu_a\in \lrs{0}{1}{r}$, i.e., $\wh{\vgu_a}(0)\ne 0$ and
\be \label{sr}
\wh{\vgu_a}(2\xi)\wh{a}(\xi)= \wh{\vgu_a}(\xi)+\bo(|\xi|^{m+1})\quad \mbox{and}\quad
\wh{\vgu_a}(2\xi)\wh{a}(\xi+\pi)=\bo(|\xi|^{m+1}),\quad \xi\to 0.
\ee
\end{enumerate}
\end{prop}

It is important to notice that \eqref{sr} for sum rules and \eqref{mfilter} for a matching filter  only depend on
the values $\{\wh{\vgu_a}^{(j)}(0)\}_{j=0}^m$ instead of $\vgu_a\in \lrs{0}{1}{r}$ itself. Under the condition in \eqref{cond:a}, using the Leibniz differentiation formula,
one can easily observe that \eqref{vgua:value} is equivalent to \eqref{mfilter}. Therefore, the existence of a matching filter $\vgu_a\in \lrs{0}{1}{r}$ in \cref{def:vsd} is often guaranteed and up to a multiplicative constant,
its values $\{\wh{\vgu_a}^{(j)}(0)\}_{j=0}^m$
are uniquely determined through \eqref{vgua:value}.

Refinable vector functions are often analyzed through vector cascade algorithms, which have been extensively studied, e.g., see \cite{cjr02,han03,hj98,hj06,jrz98} and references therein.
Recall that the refinement operator $\cd_a: (\Lp{p})^r\rightarrow (\Lp{p})^r$ is defined
in \eqref{cd} as $\cd_a f := 2 \sum_{k\in \Z} a(k) f(2\cdot-k)$ for $f\in (\Lp{p})^r$.
To study the convergence of vector cascade algorithms, we now recall a key quantity $\sm_p(a)$ which is defined in \cite{han03,hanbook}. For any sequence $a\in \lrs{0}{r}{r}$ and $v\in \lrs{0}{1}{r}$, we define
\[
\rho_{j}(a,v)_p:=
\sup\Big\{\limsup_{n\to \infty} \|(\sd_a^n (\td I_r))*u \|_{(\lp{p})^r}^{1/n} \setsp u\in \PV_{j-1,v}\Big\},\qquad j\in \NN, 1\le p\le \infty,
\]
where
\be \label{Vj}
\PV_{j-1,v}:=\{u\in (\lp{0})^r \setsp \wh{v}(\xi)\wh{u}(\xi)=\bo(|\xi|^{j}),\; \xi \to 0\},\qquad j\in \NN, v\in \lrs{0}{1}{r}.
\ee
Note that $\PV_{-1,v}:=(\lp{0})^r$ and $\PV_{j-1,v}$ is shift-invariant, that is,
$u(\cdot-k)\in \PV_{j-1,v}$ for all $u\in \PV_{j-1,v}$ and $k\in \Z$.
For any finite subset $\PB_{j-1,v}$ of elements in $\PV_{j-1,v}$ such that
$\mspan\{u(\cdot-k)\setsp k\in \Z, u\in \PB_{j-1,v}\}=\PV_{j-1,v}$ (i.e., $\PB_{j-1,v}$ generates $\PV_{j-1,v}$), it is easy to deduce that
\be \label{rho}
\rho_{j}(a,v)_p=\max\Big\{\limsup_{n\to \infty} \|(\sd_a^n(\td I_r)) *u \|_{(\lp{p})^r}^{1/n} \setsp u\in \PB_{j-1,v}\Big\}, \qquad j\in \NN, 1\le p\le \infty.
\ee

For a matrix mask $a\in \lrs{0}{r}{r}$, using the definition of sum rules in \eqref{sr},
we define $\sr(a):=m_a$ with $m_a$ being the \emph{largest} possible nonnegative integer such that the mask $a\in \lrs{0}{r}{r}$ has order $m_a$ sum rules with a matching filter $\vgu_a\in \lrs{0}{1}{r}$ and $\wh{\vgu_a}(0)\ne 0$.
Now we define the quantity $\sm_p(a)$ (see \cite[(4.3)]{han03} or \cite[(5.6.44)]{hanbook}) by
\be \label{sma}
\sm_p(a):=\frac{1}{p}-\log_2 \rho_{m_a}(a,\vgu_a)_p\quad \mbox{with}\quad m_a:=\sr(a),\qquad 1\le p\le \infty.
\ee
We shall prove in \cref{thm:smp} that the definition of $\sm_p(a)$ in \eqref{sma} is independent of the choice of a matching filter $\vgu_a$.
For $m\in \NN$ and $1\le p\le \infty$, the $L_p$ Sobolev space $\SL{m}{p}$ consists of all functions $f\in \Lp{p}$ such that all $f,\ldots, f^{(m-1)}$ are absolutely continuous functions in $\Lp{p}$ and $f^{(m)}\in L_p(\R)$. The norm on $\SL{m}{p}$ is defined to be $\|f\|_{\SL{m}{p}}:=\sum_{j=0}^m \|f^{(j)}\|_{\Lp{p}}<\infty$ for $f\in \SL{m}{p}$. For $0\le \tau\le 1$, we say that $f\in \mbox{Lip}(\tau, L_p)$ if there exists a positive constant $C$ such that $\|f-f(\cdot-t)\|_{\Lp{p}} \le C|t|^\tau$ for all $t\in \R$. For a function $f\in \Lp{p}$, we define $\sm_p(f)$ to be the supreme of $m+\tau$ with $m\in \NN$ and $0\le \tau<1$ such that $f\in \SL{m}{p}$ and $f^{(m)}\in \mbox{Lip}(\tau, L_p)$.

Now we recall the result on convergence of vector cascade algorithms in \cite{han03,hanbook} and smoothness of refinable vector functions.

\begin{theorem}\label{thm:vca}
(\cite[Theorem~4.3]{han03} or \cite[Theorem~5.6.11]{hanbook})
Let $r\in \N$, $m\in \NN$ and $1\le p\le \infty$. Let $a\in \lrs{0}{r}{r}$
and $\vgu_a\in \lrs{0}{1}{r}$ be an order $m+1$ matching filter of the mask $a$ satisfying \eqref{mfilter}. Let $\phi=[\phi_1,\ldots,\phi_r]^\tp$ be a vector of compactly supported distributions satisfying $\wh{\phi}(2\xi)=\wh{a}(\xi)\wh{\phi}(\xi)$ and $\wh{\vgu_a}(0)\wh{\phi}(0)=1$.
Then the following statements are equivalent to each other:
\begin{enumerate}
\item[(1)] The vector cascade algorithm with mask $a$ and an order $m+1$ matching filter $\vgu_a$ of the mask $a$ is $\SL{m}{p}$ convergent, i.e.,  $\{\cd_a^n f\}_{n=1}^\infty$ is a Cauchy sequence in $(\SL{m}{p})^r$ for every compactly supported vector function $f\in (\SL{m}{p})^r$ satisfying \eqref{initialf}.
\item[(2)] For one vector function $f \in (\SL{m}{p})^r$ (require $f\in (\CH{m})^r$ if $p=\infty$) such that $f$ satisfies \eqref{initialf} and the integer shifts of $f$ are stable, $\{\cd_a^n f\}_{n=1}^\infty$ is a Cauchy sequence in $(\SL{m}{p})^r$.
\item[(3)] The refinable vector function $\phi\in (\SL{m}{p})^r$ (require $\phi\in (\CH{m})^r$ if $p=\infty$) and
$\lim_{n\to \infty} \|\cd_a^n f-\phi\|_{(\SL{m}{p})^r}=0$ for every compactly supported vector function $f\in (\SL{m}{p})^r$ satisfying \eqref{initialf}.

\item[(4)] $\limsup_{n\to \infty} 2^{n(m-1/p)}\|(\sd_a^n (\td I_r))*u\|_{(\lp{p})^r}=0$ for all $u\in \PV_{m,\vgu_a}$.

\item[(5)] $\rho_{m+1}(a,\vgu_a)_p<2^{1/p-m}$.

\item[(6)] $\sm_p(a)>m$.
\end{enumerate}
Moreover, all the above items (1)--(6) must hold and $\sm_p(\phi)=\sm_p(a)$ (\cite[Corollary~5.1]{han03} or \cite[Corollary~5.6.12]{hanbook}) if
\begin{enumerate}
\item[(7)] the refinable vector function $\phi\in (\SL{m}{p})^r$ (require $\phi\in (\CH{m})^r$ if $p=\infty$) and the integer shifts of $\phi$ are stable (i.e., \eqref{stability} holds).
\end{enumerate}
Consequently, any of the above items (1)--(7) implies that $\sm_p(\phi)\ge \sm_p(a)$, $\sr(a)\ge \sm_p(a)$ and all items (1)--(3) of \cref{prop:phi} hold.
\end{theorem}

We now make some remarks about \cref{thm:vca}.
For $p=\infty$, we point out that the initial vector functions $f$ in items (1) and (3) are only required to belong to $(\SL{m}{\infty})^r$ instead of the smaller space $(\CH{m})^r$, even though the limit function $\phi=\lim_{n\to \infty} \cd_a^n f$ belongs to $(\CH{m})^r$. By \cref{prop:phi}, item (7) of \cref{thm:vca} implies that $\phi$ satisfies \eqref{vgu:phi} and hence item (2) of \cref{thm:vca} trivially holds with $f=\phi$.
The quantity $\sm_p(a)$ can be equivalently expressed through the joint spectral radius and
then $\sm_p(a)$ can be effectively estimated using
tools from joint spectral radius (e.g., see \cite{cjr02,hj98,han03,jrz98} and \cite[Section~5.7]{hanbook}).
Moreover,  the following inequalities hold:
\[
\sm_q(a) \le \sm_p(a)\le \tfrac{1}{p}-\tfrac{1}{q}+\sm_q(a),\qquad
\forall\; 1\le p\le q\le \infty.
\]
In particular, $\sm_2(a)-1/2\le \sm_\infty(a)\le \sm_2(a)$ and $\sm_2(a)$ can be computed by finding the eigenvalues of a \emph{finite} matrix, see \cite[Theorem~1.1]{jj03}, \cite[Theorem~7.1]{han03}, or \cite[Theorem~5.8.4]{hanbook} for details.
In addition, factorization of matrix masks is helpful for estimating $\sm_p(a)$, e.g., see \cite[Proposition~7.2]{han03}, \cite[Theorem~5.8.3 and Corollary~5.8.5]{hanbook} and \cite[Theorem~4.3]{han20} for details.
If the matrix mask $a$ has order $m_a$ sum rules with a matching filter $\vgu_a\in \lrs{0}{1}{r}$, using the joint spectral radius, then \cite[Theorem~5.7.6]{hanbook} shows:
\be \label{rhoj}
\rho_j(a,\vgu_a)_p=\max(2^{1/p-j}, \rho_{m_a}(a,\vgu_a)_p),\qquad \forall\, j=0,\ldots, m_a-1.
\ee
Suppose that $\sm_p(a)>m$ for some $m\in \NN$, that is,
$\rho_{m_a}(a,\vgu_a)_p=2^{1/p-\sm_p(a)}<2^{1/p-m}$ with $m_a:=\sr(a)$. It follows from \eqref{rhoj} that
$\rho_j(a,\vgu_a)_p<2^{1/p-m}$
for all $m<j\le \sr(a)$. Therefore, item (5) of \cref{thm:vca} can be replaced by $\rho_j(a,\vgu_a)_p<2^{1/p-m}$ as long as $m<j\le \sr(a)$.

Because a matching filter $\vgu_a$ of a matrix mask $a\in \lrs{0}{r}{r}$ may not be unique,
we now prove that the quantity $\sm_p(a)$ is well defined and is independent of the choice $\vgu_a$
in the following \cref{thm:smp}.
%We shall give a detailed proof for \cref{thm:smp}, largely because this technical result has never been explicitly proved in the literature.
To avoid confusion, we shall use the notation $\sm_p(a,\vgu_a)$ instead of $\sm_p(a)$ to explicitly emphasize its dependence on a matching filter $\vgu_a$.

\begin{theorem}\label{thm:smp}
Let $a\in \lrs{0}{r}{r}$ such that $a$ has order $\sr(a)$ sum rules with an order $\sr(a)$ matching filter $\vgu_a\in \lrs{0}{1}{r}$. Then $\sm_p(a,\vgu_a):=\frac{1}{p}-\log_2 \rho_{\sr(a)}(a,\vgu_a)_p$ in \eqref{sma}
is independent of the choice of the matching filter $\vgu_a\in \lrs{0}{1}{r}$ and in fact the following strengthened claim holds:
\be \label{rho:sm}
\rho_k(a,v)_p=2^{\frac{1}{p}-\sm_p(a,\vgu_a)},
\quad \mbox{or equivalently},\quad
\sm_p(a,\vgu_a)=\tfrac{1}{p}-\log_2 \rho_k(a,v)_p
\ee
for all $v\in \lrs{0}{1}{r}$ with $\wh{v}(0)\ne 0$ and $k\in \NN$ satisfying $\sm_p(a,\vgu_a)\le k \le \sr(a)$ such that the matrix mask $a$ has order $k$ sum rules with the matching filter $v$.
\end{theorem}

\bp Let $m_a:=\sr(a)$.
To prove \eqref{rho:sm}, we consider two cases $\sm_p(a,\vgu_a)>0$ or $\sm_p(a,\vgu_a)\le 0$.

Case 1: $\sm_p(a,\vgu_a)>0$. Since $\sm_p(a,\vgu_a)>0$, there exists a unique nonnegative integer $m$ such that
$m<\sm_p(a, \vgu_a)\le m+1$, i.e.,
$2^{\frac{1}{p}-m-1}\le \rho_{m_a}(a,\vgu_a)_p<2^{\frac{1}{p}-m}$.
Now it follows from \eqref{rhoj} that
\be \label{rhoj:sm:0}
\rho_j(a,\vgu_a)_p=\rho_{m_a}(a,\vgu_a)_p=
2^{\frac{1}{p}-\sm_p(a,\vgu_a)},\qquad \forall\; j=m+1,\ldots, m_a.
\ee
By \cref{thm:vca} and $\sm_p(a,\vgu_a)>m$, the condition in \eqref{cond:a} on $\wh{a}(0)$ must hold and $m_a=\sr(a)\ge m+1$.
Let $\tilde{m}$ be the largest integer such that $1$ is a simple eigenvalue of $\wh{a}(0)$ and all the other eigenvalues of $\wh{a}(0)$ are less than $2^{-\tilde{m}}$ in modulus. Hence, $\tilde{m}\ge m$ by \eqref{cond:a} and there is a nonzero constant $c$ such that $\wh{v}^{(j)}(0)=c\wh{\vgu_a}^{(j)}(0)$ for all $j=0,\ldots, \min(\tilde{m},m_a-1,k-1)$ and hence $\PV_{j,v}=\PV_{j,\vgu_a}$ for all $j=0,\ldots, \min(\tilde{m},m_a-1,k-1)$.
Because $m<\sm_p(a,\vgu_a)\le k \le m_a$, we have $m+1\le k\le m_a$ and $\min(\tilde{m},m_a-1,k-1)=\min(\tilde{m},k-1)$.
Now $\PV_{j,v}=\PV_{j,\vgu_a}$ for all $j=0,\ldots, \min(\tilde{m},k-1)$ implies
\be \label{rho:3}
\rho_{j}(a,v)_p=\rho_{j}(a,\vgu_a)_p,\qquad \forall\; j=0,\ldots,\min(\tilde{m}+1,k).
\ee
For the case $k\le \tilde{m}+1$, by $m+1\le k\le m_a$, \eqref{rho:3} and \eqref{rhoj:sm:0} together imply \eqref{rho:sm} since $\rho_k(a,v)_p=
\rho_{k}(a,\vgu_a)_p=2^{\frac{1}{p}-\sm_p(a,\vgu_a)}$.
For the case $k>\tilde{m}+1$,
by $\tilde{m}\ge m$ and $k\le m_a$,
we have $m_a\ge k\ge \tilde{m}+2\ge m+2$ and then by $m+1\le \tilde{m}+1\le m_a-1$,
we observe from \eqref{rhoj:sm:0}, \eqref{rho:3} and \eqref{rhoj} that
\be \label{rho:4}
2^{\frac{1}{p}-\sm_p(a,\vgu_a)}=
\rho_{j}(a,\vgu_a)_p
=\rho_{j}(a,v)_p=\max(2^{\frac{1}{p}-j}, \rho_k(a,v)_p),\qquad \forall\, m< j\le \tilde{m}+1.
\ee
If $\rho_k(a,v)_p\ge 2^{\frac{1}{p}-m-1}$, then \eqref{rho:sm} follows directly from \eqref{rho:4} with $j=m+1$.
If $\rho_k(a,v)_p<2^{\frac{1}{p}-m-1}$, then
we conclude from
\cite[Proposition~5.6.9]{hanbook} that $\tilde{m}\ge m+1$ and then \eqref{rho:sm} follows from $\sm_p(a,\vgu_a)\le m+1$ and \eqref{rho:4} with $j=m+2$.

Case 2: $\sm_p(a,\vgu_a)\le 0$.
Using \eqref{rhoj} if $\sr(a)>0$ or using the definition of $\sm_p(a,\vgu_a)$ if $\sr(a)=0$, we always have $\rho_0(a,\vgu_a)_p=
\rho_{m_a}(a,\vgu_a)_p=2^{\frac{1}{p}-\sm_p(a,\vgu_a)}$. Obviously, $\rho_0(a,v)_p=
\rho_0(a,\vgu_a)_p$, because $\rho_{0}(a,v)_p$ is independent of $v$ due to $\PV_{-1,v}=(\lp{0})^r$.
Hence, \eqref{rho:sm} trivially holds for $k=0$.
For $k\ge 1$, using \eqref{rhoj} again, we have
\be \label{rho:5}
2^{\frac{1}{p}-\sm_p(a,\vgu_a)}=
\rho_0(a,\vgu_a)_p=
\rho_{0}(a,v)
=\max(2^{1/p}, \rho_k(a,v)_p).
\ee
If $\rho_k(a,v)_p\ge 2^{1/p}$, then \eqref{rho:sm} follows directly from \eqref{rho:5}.
We now show that $\rho_k(a,v)_p<2^{1/p}$ cannot happen.
Suppose that $\rho_k(a,v)_p<2^{1/p}$. By \cref{thm:vca}, $\sr(a)\ge 1$ and \eqref{cond:a} must hold with $m=0$. Consequently, by \eqref{rhoj} and $\PV_{0,\vgu_a}=\PV_{0,v}$, we deduce from $\rho_k(a,v)_p<2^{1/p}$ that
\[
\rho_1(a,\vgu_a)_p=\rho_1(a,v)_p=
\max(2^{\frac{1}{p}-1}, \rho_k(a, v)_p)<2^{1/p}.
\]
Hence, we must have $\rho_{m_a}(a,\vgu_a)_p \le \rho_1(a,\vgu_a)_p<2^{1/p}$, which implies $\sm_p(a,\vgu_a)=\frac{1}{p}-\log_2 \rho_{m_a}(a,\vgu_a)_p>0$. This contradicts our assumption $\sm_a(a,\vgu_a)\le 0$.

This completes the proof of \eqref{rho:sm}.
In particular, \eqref{rho:sm} implies that $\sm_p(a)$ is well defined and is independent of the choice of a matching filter $\vgu_a$.
\ep

\section{Proofs of \cref{thm:main,thm:fastconverg,thm:mfilter}}
\label{sec:proof}

Built on the results for vector cascade algorithms and refinable vector functions stated in \cref{sec:vca}, we now prove \cref{thm:main,thm:fastconverg,thm:mfilter} on vector subdivision schemes.

Because \cref{thm:mfilter} is used in our proof of \cref{thm:main}, we first prove \cref{thm:mfilter}, which explains the scaling factor in \cref{def:vsd} for defining a vector subdivision scheme of an arbitrary matrix mask.

\begin{proof}[Proof of \cref{thm:mfilter}]
Because both $a\in \lrs{0}{r}{r}$ and $u\in (\lp{0})^r$ have finite support, there exists a positive integer $N\in \N$ such that both $a$ and $u$ are supported inside $[-N,N]$. Define $a_n$ and $u_n$ by $\wh{a_n}(\xi):=\wh{a}(2^{n-1}\xi)\wh{a}(2^{n-2}\xi)\cdots \wh{a}(2\xi)\wh{a}(\xi)$ and $\wh{u_n}(\xi):=\wh{a_n*u}(\xi)=\wh{a_n}(\xi)\wh{u}(\xi)$. Then we have $\sd_a^n (\td I_r)=2^n a_n$ and $(\sd_a^n (\td I_r))*u=2^n a_n*u=2^n u_n$. Using induction, we can easily prove that $u_n$ is supported inside $[-2^nN,2^nN]$ and therefore, we conclude from \eqref{phi:deriv:u} that $\eta$ must be supported inside $[-N,N]$ and
\eqref{phi:deriv:u} is equivalent to
\be \label{phi:deriv:bn}
\lim_{n\to \infty}
\max_{k\in \Z\cap [-2^nN,2^nN]}
\|u_n(k)2^{(\tau+1) n}-\eta(2^{-n}k)\|=0.
\ee
Since $\vgu_a$ is an order $m+1$ matching filter of the matrix mask $a$, using \eqref{mfilter}, we have
\begin{align*}
\wh{\vgu_a}(2^n\xi)\wh{u_n}(\xi)
&=\wh{\vgu_a}(2^n\xi) \wh{a_n}(\xi)\wh{u}(\xi)
=
\wh{\vgu_a}(2^n\xi) \wh{a}(2^{n-1}\xi)\cdots \wh{a}(2\xi)\wh{a}(\xi)\wh{u}(\xi)\\
&=\wh{\vgu_a}(2^{n-1}\xi) \wh{a}(2^{n-2}\xi)\cdots \wh{a}(2\xi)\wh{a}(\xi)\wh{u}(\xi)+\bo(|\xi|^{m+1})
=\cdots=\wh{\vgu_a}(\xi)\wh{u}(\xi)+\bo(|\xi|^{m+1})
\end{align*}
as $\xi \to 0$. In other words, the above identities are equivalent to saying that
\be \label{vguau:1}
[\wh{\vgu_a}\wh{u}]^{(\ell)}(0)=
[\wh{\vgu_a}(2^n\cdot) \wh{u_n}(\cdot)]^{(\ell)}(0)
=\sum_{s=0}^\ell \frac{\ell!}{s!(\ell-s)!}
\wh{\vgu_a}^{(\ell-s)}(0)
2^{(\ell-s)n} \wh{u_n}^{(s)}(0),\quad
\ell=0,\ldots,m.
\ee
Noting that $\wh{\eta}^{(s)}(0)=\int_\R \eta(x) (-ix)^s dx$,
we now use \eqref{phi:deriv:bn} to
claim that
\be \label{vguau:2}
\lim_{n\to \infty}
2^{(\ell-s)n} \wh{u_n}^{(s)}(0)=\lim_{n\to \infty}2^{(\ell-\tau)n}
\wh{\eta}^{(s)}(0)=\td(\ell-\tau) \wh{\eta}^{(s)}(0),\qquad \forall\, \ell, s\in \NN, \ell \le \tau.
\ee
Because $\eta$ is continuous and is supported inside $[-N,N]$, using the Riemann sum for $\wh{\eta}^{(s)}(0)=\int_\R \eta(x) (-ix)^s dx$, we have
\be \label{eta:to0}
\lim_{n\to \infty} I^{s}_{n}=0
\quad \mbox{with}\quad
I^{s}_{n}:=\left(2^{-n}\sum_{k=-2^nN}^{2^nN}
\eta(2^{-n}k) (-i2^{-n}k)^s\right)-
\int_{\R} \eta(x)(-ix)^s dx.
\ee
Note that $\wh{u_n}^{(s)}(0)=
\sum_{k=-2^nN}^{2^n N}
u_n(k)(-ik)^s$ and $\wh{\eta}^{(s)}(0)=\int_\R \eta(x) (-ix)^s dx$.
Hence, for $s\in \NN$,
\begin{align*}
\left| 2^{(\ell-s)n}\wh{u_n}^{(s)}(0)-2^{(\ell-\tau)n}\wh{\eta}^{(s)}(0)
\right|
&\le 2^{(\ell-s-\tau-1)n}\left| \sum_{k=-2^nN}^{2^nN}
[u_n(k) 2^{(\tau+1)n}-\eta(2^{-n}k)] (-ik)^s \right|+2^{(\ell-\tau)n} |I^s_n|\\
&\le
2^{(\ell-s-\tau-1)n}
\|u_n(\cdot) 2^{(\tau+1)n}-\eta(2^{-n}\cdot)\|_{(\lp{\infty})^r}
\sum_{k=-2^nN}^{2^nN} |k|^s+
2^{(\ell-\tau)n}|I^s_n|\\
&\le C_{N,s} 2^{(\ell-\tau)n} \|u_n(\cdot) 2^{(\tau+1)n}-\eta(2^{-n}\cdot)\|_{(\lp{\infty})^r}
+2^{(\ell-\tau)n}|I^s_n|,
\end{align*}
where $C_{N,s}$ is a positive constant independent of $n$ such that $\sum_{k=-2^nN}^{2^nN} |k|^s\le C_{N,s}2^{(s+1)n}$ for all $n\in \N$. Using \eqref{phi:deriv:bn} and \eqref{eta:to0}, we conclude from the above inequality that \eqref{vguau:2} holds. This proves \eqref{vguau:2}.

Now for $\ell\in \{0,\ldots,m\}$ and $\ell\le \tau$, we deduce from \eqref{vguau:1} and \eqref{vguau:2} that
\[
[\wh{\vgu_a}\wh{u}]^{(\ell)}(0)
=\sum_{s=0}^\ell \frac{\ell!}{s! (\ell-s)!}
\wh{\vgu_a}^{(\ell-s)}(0) \lim_{n\to \infty}
2^{(\ell-s)n} \wh{u_n}^{(s)}(0)
=\lim_{n\to \infty} 2^{(\ell-\tau)n} \sum_{s=0}^\ell \frac{\ell!}{s! (\ell-s)!}
\wh{\vgu_a}^{(\ell-s)}(0) \wh{\eta}^{(s)}(0),
\]
from which we conclude that
\be \label{vgua:u:0}
[\wh{\vgu_a}\wh{u}]^{(\ell)}(0)=
\begin{cases}
0, &\text{if $\ell\in \{0,\ldots,m\}$ and $\ell<\tau$},\\
[\wh{\vgu_a}\wh{\eta}]^{(\ell)}(0), &\text{if $\ell\in \{0,\ldots, m\}$ and $\ell=\tau\in \NN$}.
\end{cases}
\ee
To prove item (1), we next prove the second part of \eqref{limit:eta}.
Because $\eta$ is continuous and is supported inside $[-N,N]$, using the Riemann sum for $\wh{\eta}(\xi)=\int_\R \eta(x) e^{-i\xi x} dx$, we have
\be \label{hateta}
\lim_{n\to \infty} I_n(\xi)=0,\qquad \forall\, \xi\in \R\quad \mbox{with}\quad
I_n(\xi):=\wh{\eta}(\xi)-
2^{-n}\sum_{k=-2^nN}^{2^nN}
\eta(2^{-n}k) e^{-i 2^{-n} k \xi}.
\ee
Since $u_n$ is supported inside $[-2^nN, 2^nN]$, we have $\wh{u_n}(2^{-n}\xi)
=\sum_{k=-2^nN}^{2^nN}
u_n(k) e^{-i 2^{-n} k\xi}$. Consequently,
\begin{align*}
|2^{\tau n} \wh{u_n}(2^{-n}\xi)-
\wh{\eta}(\xi)|
&\le 2^{-n}\left|\sum_{k=-2^nN}^{2^nN}
[u_n(k)2^{(\tau+1) n}-\eta(2^{-n}k)] e^{-i 2^{-n} k\xi}\right|
+|I_n(\xi)|\\
&\le 2^{-n} (2^{n+1}N+1) \| u_n(\cdot) 2^{(\tau+1) n}-\eta(2^{-n}\cdot)\|_{(\lp{\infty})^r}+|I_n(\xi)|.
\end{align*}
Using \eqref{phi:deriv:bn} and \eqref{hateta}, we see from the above inequalities that the second part of \eqref{limit:eta} holds, because
$\left(\prod_{j=1}^{n} \wh{a}(2^{-j}\xi)\right) \wh{u}(2^{-n}\xi)=
\wh{a_n}(2^{-n}\xi)\wh{u}(2^{-n}\xi)=
\wh{u_n}(2^{-n}\xi)$.
From the second part of \eqref{limit:eta}, we have
\begin{align*}
\wh{\eta}(2\xi)
&=
\lim_{n\to \infty} \left[2^{\tau n} \left(\prod_{j=1}^{n} \wh{a}(2^{1-j}\xi)\right) \wh{u}(2^{1-n}\xi)\right]\\
&=\lim_{n\to \infty} 2^{\tau} \wh{a}(\xi) \left[2^{\tau(n-1)} \left(\prod_{j=1}^{n-1} \wh{a}(2^{-j}\xi)\right) \wh{u}(2^{-(n-1)}\xi)\right]=
2^{\tau} \wh{a}(\xi)\wh{\eta}(\xi).
\end{align*}
This proves the first part of \eqref{limit:eta}.
To prove \eqref{phi:deriv:uc},
since $c\in \lp{0}$ and $\wh{c}(0)=\sum_{k\in \Z} c(k)$, we have
\begin{align*}
&\left\| [(\sd_a^n (\td I_r))*(u*c)](\cdot)2^{\tau n}-
\wh{c}(0) \eta(2^{-n}\cdot)
\right\|_{(\lp{\infty})^{r}}
\\ &\qquad
\le \Big\| \sum_{k\in \Z} c(k)\Big([(\sd_a^n (\td I_r))*u](\cdot-k)2^{\tau n}-  \eta(2^{-n}(\cdot-k))\Big)
\Big\|_{(\lp{\infty})^{r}}
\\ &\qquad \qquad +
\Big\| \sum_{k\in \Z} c(k) \Big(\eta(2^{-n}(\cdot-k))-\eta(2^{-n}\cdot)
\Big)\Big\|_{(\CH{})^r}\\
&\qquad \le \|c\|_{\lp{1}}
\| [(\sd_a^n (\td I_r))*u](\cdot)2^{\tau n}- \eta(2^{-n}\cdot)
\|_{(\lp{\infty})^{r}}
%\\ &\qquad\qquad
+ \sum_{k\in \Z} |c(k)| \| \eta(\cdot-2^{-n}k)-\eta(\cdot)
\|_{(\CH{})^r}.
\end{align*}
Since $c\in \lp{0}$ and $\lim_{n\to \infty} \| \eta(\cdot-2^{-n}k)-\eta(\cdot)
\|_{(\CH{})^r}=0$ for all $k\in \Z$, we conclude from \eqref{phi:deriv:u} and the above inequality that \eqref{phi:deriv:uc} holds. This proves all the claims in item (1).

We now prove item (2).
Since $\eta$ has compact support and is not identically zero, $\wh{\eta}$ is an analytic vector function and hence, there must exist $j\in \NN$ such that
\be \label{eta:j:0}
\wh{\eta}(0)=\cdots=\wh{\eta}^{(j-1)}(0)=0,\quad
\wh{\eta}^{(j)}(0)\ne 0,
\quad\mbox{and}\quad \wh{\eta}(\xi)=\frac{\wh{\eta}^{(j)}(0)}{j!} \xi^j+\bo(|\xi|^{j+1}),\quad \xi \to 0.
\ee
From the proved fact $\wh{\eta}(2\xi)=2^\tau \wh{a}(\xi)\wh{\eta}(\xi)$ in \eqref{limit:eta}, we deduce that
\[
2^{j}
\frac{\wh{\eta}^{(j)}(0)}{j!} \xi^j=\wh{\eta}(2\xi)+\bo(|\xi|^{j+1})
=2^\tau \wh{a}(\xi)\wh{\eta}(\xi)+\bo(|\xi|^{j+1})=
2^\tau \wh{a}(0) \frac{\wh{\eta}^{(j)}(0)}{j!} \xi^j+\bo(|\xi|^{j+1}),\quad \xi \to 0.
\]
Because $\wh{\eta}^{(j)}(0)\ne 0$, the above identity forces
\be \label{eta:tau} \wh{a}(0)\wh{\eta}^{(j)}(0)=2^{j-\tau}
\wh{\eta}^{(j)}(0) \ne 0.
\ee
That is, $2^{j-\tau}$ must be an eigenvalue of $\wh{a}(0)$.
Since $\tau\in [0,m]$ and $j\ge 0$, we have $j-\tau\ge -m$ and hence,
$2^{j-\tau}\ge 2^{-m}$. Now we deduce from the condition in \eqref{cond:a} that $j-\tau=0$, i.e.,
$\tau=j$ must be an integer.
Because $1$ is a simple eigenvalue of $\wh{a}(0)$, we conclude from
$\wh{\vgu_a}(0)\wh{a}(0)=\wh{\vgu_a}(0)$ with $\wh{\vgu_a}(0)\ne 0$ and \eqref{eta:tau} with $\tau=j$ that $\wh{\vgu_a}(0)\wh{\eta}^{(j)}(0)\ne 0$. Moreover, we observe from
\eqref{vgua:u:0} and \eqref{eta:j:0} that
\[
[\wh{\vgu_a}\wh{u}]^{(j)}(0)=[\wh{\vgu_a}\wh{\eta}]^{(j)}(0)
=\sum_{\ell=0}^j \frac{j!}{\ell!(j-\ell)!}
\wh{\vgu_a}^{(j-\ell)}(0)\wh{\eta}^{(\ell)}(0)
=\wh{\vgu_a}(0)\wh{\eta}^{(j)}(0) \ne 0.
\]
This proves $\tau=j=\ld_m(\wh{\vgu_a}\wh{u})$ and $\beta_j:=\frac{[\wh{\vgu_a}\wh{u}]^{(j)}(0)}{i^j j!}\ne 0$.

By \eqref{eta:j:0}, we have $\wh{\eta}(0)=\cdots=\wh{\eta}^{(j-1)}(0)=0$. Therefore, we can recursively define vector functions $\eta_\ell(x):=\int_{-\infty}^x \eta_{\ell-1}(t) dt$ for $\ell=1,\ldots,j$ with $\eta_0:=\eta$.
The identities $\wh{\eta}(0)=\cdots=\wh{\eta}^{(j-1)}(0)=0$  guarantee that all $\eta_1,\ldots,\eta_j$ are compactly supported vector functions, $\eta_\ell \in (\CH{\ell})^r$ and
$\eta_\ell^{(\ell)}=\eta$
for all $\ell=0,\ldots,j$. In particular, $\eta_j\in (\CH{j})^r$ and $\eta_j^{(j)}=\eta$.
Consequently, $\wh{\eta}(\xi)=(i\xi)^j \wh{\eta_j}(\xi)$, from which we have $\wh{\eta}^{(j)}(0)=i^j j! \wh{\eta_j}(0)$.
Because $\wh{\eta}(2\xi)=2^j \wh{a}(\xi)\wh{\eta}(\xi)$ and $\wh{\eta}(\xi)=(i\xi)^j\wh{\eta_j}(\xi)$, we deduce that
\[
2^j (i\xi)^j \wh{\eta_j}(2\xi)
=(i2\xi)^j \wh{\eta_j}(2\xi)
=\wh{\eta}(2\xi)=2^j \wh{a}(\xi)\wh{\eta}(\xi)=
2^j (i\xi)^j \wh{a}(\xi)\wh{\eta_j}(\xi),
\]
from which we must have $\wh{\eta_j}(2\xi)= \wh{a}(\xi)\wh{\eta_j}(\xi)$.
By \eqref{cond:a} and the uniqueness of the refinable vector function $\phi$, there must exist $c\in \C$ such that $\eta_j=c\phi$. Hence, $\eta=\eta_j^{(j)}=c\phi^{(j)}$. We now determine the constant $c$. By $\wh{\vgu_a}(0)\wh{\phi}(0)=1$ and $\eta_j=c\phi$, we have
\[
c=c\wh{\vgu_a}(0)\wh{\phi}(0)
=\wh{\vgu_a}(0)\wh{\eta_j}(0)
=\frac{1}{i^j j!} \wh{\vgu_a}(0)\wh{\eta}^{(j)}(0)
=\frac{[\wh{\vgu_a}\wh{u}]^{(j)}}{i^j j!}=\beta_j \ne 0,
\]
where we used  $\wh{\eta}^{(j)}(0)=i^j j! \wh{\eta_j}(0)$ and $[\wh{\vgu_a}\wh{u}]^{(j)}(0)=\wh{\vgu_a}(0)\wh{\eta}^{(j)}(0)\ne 0$.
This proves item (2).

Finally, we prove item (3) using proof by contradiction. Suppose that $\eta$ is identically zero. By $\tau=j$ and $\sd_a^n (\td I_r)=2^n a_n$, \eqref{phi:deriv} becomes
$\lim_{n\to \infty}
2^{(j+1)n} \|a_n*u\|_{(\lp{\infty})^r}=0$.
By \cite[Proposition~4.1]{han03} or
\cite[Proposition~5.6.9 or Theorem~5.7.5]{hanbook}, this limit holds if and only if there exists $\varepsilon>0$ such that
\be \label{sa:u:to0}
\lim_{n\to \infty}
\|[(\sd_a^n(\td I_r))*u](\cdot) 2^{(j+\varepsilon) n}\|_{(\lp{\infty})^r}=
\lim_{n\to \infty}
2^{(j+1+\varepsilon)n} \|a_n*u\|_{(\lp{\infty})^r}=0.
\ee
By item (1), we deduce from \eqref{vgua:u:0} with $\tau:=j+\varepsilon>j$ that $[\wh{\vgu_a}\wh{u}]^{(j)}(0)=0$, a contradiction to our assumption $\beta_j\ne 0$. This proves that $\eta$ cannot be identically zero. By item (1) and $\beta_j\ne 0$, we must have $j=\ld_m(\wh{\vgu_a}\wh{u})$.
By the same argument as in proving item (2), $\wh{\eta}(0)=\wh{\eta}'(0)=\cdots=\wh{\eta}^{(j-1)}(0)=0$. As in the proof of item (2),
we define $\eta_\ell, \ell=0,\ldots,j$. Then we already proved $\eta_j\in (\CH{j})^r$, $\eta=\eta_j^{(j)}$, $\wh{\eta_j}(2\xi)=\wh{a}(\xi)\wh{\eta_j}(\xi)$ and
$i^j j! \wh{\eta_j}(0)=\wh{\eta}^{(j)}(0)\ne 0$. By \eqref{vgua:u:0} and \eqref{eta:j:0}, we also have
\[
i^j j! \wh{\vgu_a}(0)\wh{\eta_j}(0)
=\wh{\vgu_a}(0)\wh{\eta}^{(j)}(0)
=[\wh{\vgu_a}\wh{\eta}]^{(j)}(0)=
[\wh{\vgu_a}\wh{u}]^{(j)}(0)=i^j j! \beta_j\ne 0.
\]
Hence, $\beta_j=\wh{\vgu_a}(0)\wh{\eta_j}(0)$.
Now define $\varphi:=\beta_j^{-1} \eta_j$. Then
$\eta=\eta_j^{(j)}=\beta_j \varphi^{(j)}$, $\wh{\varphi}(2\xi)=\wh{a}(\xi)\wh{\varphi}(\xi)$,
and
$\wh{\vgu_a}(0)\wh{\varphi}(0)=
\beta_j^{-1} \wh{\vgu_a}(0) \wh{\eta_j}(0)=1$.
This completes the proof of item (3).
\end{proof}

To prove \cref{thm:main,thm:fastconverg}, we need the following auxiliary result, which is of interest in itself.

\begin{lemma}\label{lem:phi:Z}
Let $r, m_a\in \N$ and $m\in \NN$ with $m_a\ge m+1$.
Let $\vgu_a\in \lrs{0}{1}{r}$ with $\wh{\vgu_a}(0)\ne 0$ and
$\phi$ be an $r\times 1$ vector of compactly supported functions  satisfying
\be \label{phi:poly}
\wh{\vgu_a}(\xi)\wh{\phi}(\xi+2\pi k)=
\td(k)+\bo(|\xi|^{m_a}),\quad \xi \to 0, \; \forall\; k\in \Z.
\ee
If $\phi\in (\CH{m})^r$,
then the sequences $u_{\phi,j}\in (\lp{0})^r, j=0,\ldots,m$ defined by
\be \label{wj}
u_{\phi,j}:=\phi^{(j)}|_{\Z}, \quad \mbox{that is},\quad
u_{\phi,j}(k):=\phi^{(j)}(k),\qquad \forall\; k\in \Z, j=0,\ldots,m
\ee
must satisfy
\be \label{vgua:wj}
\wh{\vgu_a}(\xi)\wh{u_{\phi,j}}(\xi)=(i\xi)^j+\bo(|\xi|^{m_a}),\qquad \xi \to 0, \; \forall\; j=0,\ldots,m.
\ee
\end{lemma}

\begin{proof}
By \eqref{phi:poly}, it is known (e.g., see \cite[Theorem~5.5.1]{hanbook}) that
$\sum_{k\in \Z} (\pp*\vgu_a)(k)\phi(\cdot-k)=\pp$ for all polynomials $\pp\in \PL_{m_a-1}$. In particular, we have $\pp_{m_a-1}=\sum_{k\in \Z}(\pp_{m_a-1}*\vgu_a)(k)\phi(\cdot-k)$, where
$\pp_\ell(x):=\frac{x^\ell}{\ell !}$ for $\ell\in \NN$. By $\phi\in (\CH{m})^r$,
for $j=0,\ldots,m$, we have
$\pp_{m_a-j-1}=[\pp_{m_a-1}]^{(j)}=
\sum_{k\in \Z} (\pp_{m_a-1}*\vgu_a)(k)\phi^{(j)}(\cdot-k)$. Consequently, by the definition $u_{\phi,j}=\phi^{(j)}|_{\Z}$, we have
\[
\pp_{m_a-j-1}(\cdot)=\sum_{k\in \Z} (\pp_{m_a-1}*\vgu_a)(k)u_{\phi,j}(\cdot-k)
=(\pp_{m_a-1}*\vgu_a)*u_{\phi,j}
=\pp_{m_a-1}*(\vgu_a*u_{\phi,j}).
\]
Now applying \cite[Lemma~1.2.1]{hanbook} and noting that $\pp_{m_a-1}^{(k)}=\pp_{m_a-k-1}$ for $0\le k<m_a$, we have
\[
\pp_{m_a-j-1}=
\pp_{m_a-1}*(\vgu_a*u_{\phi,j})
=\sum_{k=0}^{m_a-1} \frac{(-i)^k}{k!}
\pp_{m_a-1}^{(k)}(\cdot) [\wh{\vgu_a}\wh{u_{\phi,j}}]^{(k)}(0)
=\sum_{k=0}^{m_a-1}
\pp_{m_a-k-1}(\cdot)\frac{1}{i^k k!}[\wh{\vgu_a}\wh{u_{\phi,j}}]^{(k)}(0).
\]
Since all $\pp_k, 0\le k< m_a$ are linearly independent, we conclude from the above identity that $[\wh{\vgu_a}\wh{u_{\phi,j}}]^{(j)}(0)=i^j j!$ and $[\wh{\vgu_a}\wh{u_{\phi,j}}]^{(k)}(0)=0$ for all $k\in \{0,\ldots,m_a-1\}\bs\{j\}$, which is just \eqref{vgua:wj}.
\end{proof}

Let $\phi\in (\CH{m})^r$ be a compactly supported refinable vector function satisfying $\wh{\phi}(2\xi)=\wh{a}(\xi)\wh{\phi}(\xi)$ for some finitely supported matrix mask $a\in \lrs{0}{r}{r}$.
Because $\phi^{(j)}=2^{j+1}\sum_{k\in \Z} a(k) \phi^{(j)}(2\cdot-k)$, the finitely supported sequence $u_{\phi,j}\in (\lp{0})^r$ in \eqref{wj} (i.e.,
$u_{\phi,j}:=\phi^{(j)}|_{\Z}$)
satisfies $\tz_a u_{\phi,j}=2^{-j} u_{\phi,j}$, i.e., the finitely supported sequence $u_{\phi,j}$ is an eigenvector of $\tz_a$ for the eigenvalue $2^{-j}$, where
$\tz_a: (\lp{0})^r\rightarrow (\lp{0})^r$ is the transition operator defined by
\be \label{tz}
(\tz_a u)(j):=2\sum_{k\in \Z} a(k)u(2j-k),\qquad u\in (\lp{0})^r.
\ee
Because $\phi=\cd_a^n \phi=\sum_{k\in \Z} (\sd_a^n (\td I_r))(k) \phi(2^n\cdot-k)$, we have
$\phi^{(j)}(2^{-n}\cdot)=2^{jn} \sum_{k\in \Z} (\sd_a^n (\td I_r))(k) \phi^{(j)}(\cdot-k)$.
Therefore, it is easy to conclude that $\phi^{(j)}(2^{-n}\cdot)|_{\Z}=2^{jn} (\sd_a^n(\td I_r))*u_{\phi,j}$ for all $n\in \N$. Moreover, if $\sm_\infty(a)>j$, then $2^{-j}$ is a simple eigenvalue of $\tz_a$ and the sequence $u_{\phi,j}\in (\lp{0})^r$ in \eqref{wj} is the unique eigenvector of $\tz_a$ satisfying $\tz_a u_{\phi,j}=2^{-j} u_{\phi,j}$ and $[\wh{\vgu_a}\wh{u_{\phi,j}}]^{(j)}(0)=i^j j!$.

We are now ready to prove \cref{thm:main}.

\begin{proof}[Proof of \cref{thm:main}]
(1)$\imply$(2).
Let $m_a:=\sr(a)$. By \cref{thm:vca} and item (1), we have $\sr(a)\ge \sm_\infty(a)>m$ and \eqref{phi:poly} must hold. Consequently, $m_a\ge m+1$.
Since $\sm_\infty(a)>m$ implies $\phi\in (\CH{m})^r$, we can define
$u_{\phi,j}:=\phi^{(j)}|_{\Z}$ for $j=0,\ldots,m$, i.e., $u_{\phi,j}(k):=\phi^{(j)}(k)$ for all $k\in \Z$. Then it follows from \cref{lem:phi:Z}
that \eqref{vgua:wj} holds with $m_a\ge m+1$.

Note that $\wh{\phi}(2\xi)=\wh{a}(\xi)\wh{\phi}(\xi)$ is equivalent to $\cd_a \phi=\phi$. Therefore, $\phi=\cd_a^n \phi$ for all $n\in \N$ and hence,
\[
\phi=\cd_a^n \phi=
\sum_{\ell \in \Z} (\sd_a^n (\td I_r))(\ell) \phi(2^n\cdot-\ell)=
2^n \sum_{\ell \in \Z} a_n(\ell)\phi(2^n\cdot-\ell),
\]
where $a_n:=2^{-n}\sd_a^n (\td I_r)$, i.e., $\wh{a_n}(\xi)=\wh{a}(2^{n-1}\xi)\cdots \wh{a}(2\xi)\wh{a}(\xi)$. Since $\phi\in (\CH{m})^r$, now we deduce from the above identity that for $j=0,\ldots,m$,
\be \label{phi:j}
\phi^{(j)}(2^{-n} k)=2^{(j+1)n}
\sum_{\ell\in \Z} a_n(\ell) \phi^{(j)}(k-\ell)=2^{(j+1)n} [a_n*u_{\phi,j}](k),\qquad k\in \Z.
\ee
For a given $u\in (\lp{0})^r$, by the definition of $j:=\ld_m(\wh{\vgu_a}\wh{u})$ and $\beta_j:=\frac{[\wh{\vgu_a}\wh{u}]^{(j)}(0)}{i^j j!}$ in \eqref{phi:deriv}, we have
$\wh{\vgu_a}(\xi)\wh{u}(\xi)=\beta_j (i\xi)^j+\bo(|\xi|^{j+1})$ as $\xi \to 0$.
Using \eqref{vgua:wj} with $m_a\ge m+1$ and $\wh{\vgu_a}(\xi)\wh{u}(\xi)=\beta_j (i\xi)^j+\bo(|\xi|^{j+1})$ as $\xi \to 0$,
we have
\[
\wh{\vgu_a}(\xi)[\wh{u}(\xi)-\beta_j \wh{u_{\phi,j}}(\xi)]=
\wh{\vgu_a}(\xi)\wh{u}(\xi)-
\beta_j \wh{\vgu_a}(\xi)\wh{u_{\phi,j}}(\xi)
=\beta_j(i\xi)^j-\beta_j (i\xi)^j+\bo(|\xi|^{j+1})
=\bo(|\xi|^{j+1})
\]
as $\xi \to 0$ by $j+1\le m+1\le m_a$.
Therefore, $u-\beta_j u_{\phi,j}\in \PV_{j,\vgu_a}$.
Because $\sm_\infty(a)>m\ge j$, we conclude from item (5) of \cref{thm:vca} that
$\limsup_{n\to \infty} 2^{jn} \|(\sd_a^n(\td I_r))*w\|_{(\lp{\infty})^r}=0$ for all $w\in \PV_{j,\vgu_a}$. In particular, it follows from $u-\beta_j u_{\phi,j}\in \PV_{j,\vgu_a}$ that
\[
\limsup_{n\to \infty} 2^{jn} \|(\sd_a^n(\td I_r))*( u-\beta_j u_{\phi,j}) \|_{(\lp{\infty})^r}=0.
\]
Now it follows trivially from \eqref{phi:j} and the above limit that \eqref{phi:deriv} holds. This proves (1)$\imply$(2).

(2)$\imply$(3). Let $w_0\in \lrs{}{1}{r}$ and $u\in (\lp{0})^r$.
Define $\eta:=w_0*\phi:=\sum_{k\in \Z} w_0(k)\phi(\cdot-k)$. For any positive integer $K\in \N$, we now show that
\eqref{phi:deriv} must imply \eqref{vsd:converg}.
Because both $a$ and $u$ are finitely supported, there exists $N\in \N$ such that both $a$ and $u$ are supported inside $[-N,N]$. Define $u_n:=a_n*u$ and $w_n:=\sd_a^n w_0$.
Because the Fourier series of $u_n$ is $\wh{a}(2^{n-1}\xi)\wh{a}(2^{n-2}\xi)\cdots \wh{a}(2\xi)\wh{a}(\xi)\wh{u}(\xi)$, we can easily verify that $u_n$ is supported inside $[-2^n N, 2^n N]$.
Define $w_{0,n}\in \lrs{}{1}{r}$ by $w_{0,n}(2^n k):=w_0(k)$ for all $k\in \Z$ and $w_{0,n}(k)=0$ for all $k\not\in \Z\bs [2^n \Z]$. Then
\[
w_n*u=(\sd_a^n w_0)*u=
2^n (w_{0,n}*a_n)*u=
2^n w_{0,n}*(a_n*u)
=2^n \sum_{\ell\in \Z} w_0(\ell)u_n(\cdot-2^n \ell).
\]
Since $u_n$ is supported inside $[-2^nN, 2^nN]$, for $k\in \Z\cap [-2^n K, 2^n K]$, we have
\[
(w_n*u)(k)=2^n \sum_{\ell\in \Z} w_0(\ell)u_n(k-2^n \ell)
=2^n \sum_{\ell=-N-K}^{N+K}
w_0(\ell) u_n(k-2^n \ell).
\]
Note that $\phi$ is also supported inside $[-N,N]$. Hence,
for $k\in \Z\cap [-2^n K, 2^nK]$, we must have
\[
\eta^{(j)}(2^{-n} k)=\sum_{\ell \in \Z} w_0(\ell)\phi^{(j)}(2^{-n} k-\ell)=\sum_{\ell=-N-K}^{N+K}
w_0(\ell)\phi^{(j)}(2^{-n}(k-2^n \ell)).
\]
Consequently, by $w_n=\sd_a^n w_0$ and $w_n*u=2^n w_{0,n}*u_n$, we deduce from the above two identities that
\[
\max_{k\in \Z \cap [-2^n K, 2^n K]} |[(\sd_a^n w_0)*u](k) 2^{jn}-
\beta_j \eta^{(j)}(2^{-n}k)|
\le C_{K,N,w_0} \| [(\sd_a^n (\td I_r))*u](\cdot) 2^{jn}-\beta_j \phi^{(j)}(2^{-n}\cdot)\|_{(\lp{\infty})^r},
\]
where $C_{K,N,w_0}:=\sum_{\ell=-N-K}^{N+K}\|w_0(\ell)\|<\infty$ and we used $u_n=a_n*u=2^{-n}(\sd_a^n (\td I_r))*u$.
Now we conclude from \eqref{phi:deriv} and the above inequality that \eqref{vsd:converg} holds.
This proves (2)$\imply$(3).

(3)$\imply$(4) is trivial. We now prove (4)$\imply$(5).
Let $u\in \PB$. If $[\wh{\vgu_a}\wh{u}]^{(m)}(0)=0$, then item (5) follows directly from item (4). Suppose that $[\wh{\vgu_a}\wh{u}]^{(m)}(0)\ne 0$. By item (4), for $w_0=\td I_r$, there exists a vector function $\vec{\eta}\in (\CH{m})^r$ such that
$\lim_{n\to \infty} \| [(\sd_a^n (\td I_r)*u](\cdot)2^{mn} - \vec{\eta}^{(m)}(2^{-n}\cdot)\|_{(\lp{\infty})^r}=0$.
By item (3) of \cref{thm:mfilter}, we must have $\varphi\in (\CH{m})^r$ and $\vec{\eta}=\beta_m \varphi^{(m)}$, where $\varphi$ is a compactly supported refinable vector function satisfying $\wh{\varphi}(2\xi)=\wh{a}(\xi)\wh{\varphi}(\xi)$ and $\wh{\vgu_a}(0)\wh{\varphi}(0)=1$.
Because $1$ is a simple eigenvalue of $\wh{a}(0)$ and all $2^k, k\in \N$ are not eigenvalues of $\wh{a}(0)$, we must have the uniqueness of the refinable vector function $\phi$ and hence $\varphi=\phi$. This proves $\phi\in (\CH{m})^r$ and hence (4)$\imply$(5).

(5)$\imply$(1).
For any $u\in \PV_{m,\vgu_a}$,
by our assumption on $\PB$, we must have
$u\in \mspan\{ w(\cdot-k) \setsp w\in \PB, k\in \Z\}$ and hence we can write $u=u_1*c_1+\cdots+u_s*c_s$ with $u_1,\ldots, u_s\in \PB\subseteq \PV_{m-1,\vgu_a}$ and $c_1,\ldots, c_s\in \lp{0}$.
By item (5) and \eqref{phi:deriv:uc}, for all $\ell=1,\ldots,s$, we have
\[
\lim_{n\to \infty} \| [(\sd_a^n (\td I_r))*(u_\ell*c_\ell)](\cdot)2^{mn}-\beta_\ell \wh{c_\ell}(0) \phi^{(m)}(2^{-n}\cdot)
\|_{(\lp{\infty})^{r}}=0,\qquad \ell=1,\ldots,s,
\]
where $\beta_\ell:=\frac{[\wh{\vgu_a}\wh{u_\ell}]^{(m)}(0)}{i^j j!}$.
Hence, for $u=u_1*c_1+\cdots+u_s*c_s$, the above identity implies
\be \label{phi:deriv:Bus}
\lim_{n\to \infty} \| [(\sd_a^n (\td I_r))*(u_1*c_1+\cdots+u_s*c_s)](\cdot)2^{mn}-(\beta_1 \wh{c_1}(0)+\cdots+\beta_s\wh{c_s}(0)) \phi^{(m)}(2^{-n}\cdot)
\|_{(\lp{\infty})^{r}}=0.
\ee
By the definition of $\PV_{j-1,\vgu_a}$ in \eqref{Vj} and $u\in \PV_{m,\vgu_a}$, we must have
\begin{align*}
0&=[\wh{\vgu_a}\wh{u}]^{(m)}(0)
=[\wh{c_1}\wh{\vgu_a}\wh{u_1}]^{(m)}(0)+
\cdots+[\wh{c_s} \wh{\vgu_a}\wh{u_s}]^{(m)}(0)\\
&=\wh{c_1}(0)[\wh{\vgu_a}\wh{u_1}]^{(m)}(0)+
\cdots+\wh{c_s}(0)[\wh{\vgu_a}\wh{u_s}]^{(m)}(0)
=i^m m! \left[\wh{c_1}(0)\beta_1+\cdots+\wh{c_s}(0)\beta_s\right].
\end{align*}
That is, $\beta_1\wh{c_1}(0)+\cdots+\beta_s\wh{c_s}(0)=0$.
Hence, we conclude from \eqref{phi:deriv:Bus} that
$\lim_{n\to \infty} 2^{mn} \|(\sd_a^n (\td I_r))*u
\|_{(\lp{\infty})^{r}}=0$ for all $u\in \PV_{m,\vgu_a}$. This proves that item (4) of \cref{thm:vca} is satisfied. Consequently, by \cref{thm:vca}, we conclude that $\sm_\infty(a)>m$. This proves (5)$\imply$(1).

By \cite[Corollary~5.1]{han03} or
\cite[Corollary~5.6.12]{hanbook}, item (6) implies $\sm_\infty(a)>m$ and hence all items (1)--(5) must hold. Indeed, by item (6), $\phi\in (\CH{m})^r$ and $\mspan\{\wh{\phi}(2\pi k) \setsp k\in \Z\}=\C^r$. Now by \cref{prop:phi}, $\phi$ is an admissible initial function satisfying \eqref{vgu:phi} and \eqref{stability}. Hence, item (2) of \cref{thm:vca} trivially holds by taking $f=\phi$. Hence, by \cref{thm:vca}, item (6) implies $\sm_\infty(a)>m$.
\end{proof}

We now prove \cref{thm:fastconverg} for convergence rates of vector subdivision schemes.

\begin{proof}[Proof of \cref{thm:fastconverg}]
Let $m_a:=\sr(a)$. By \cref{thm:vca} and $\sm_\infty(a)>m$, we have $\sr(a)\ge \sm_\infty(a)>m$ and \eqref{phi:poly} must hold. Consequently, $m_a\ge m+1$.
Since $\sm_\infty(a)>m$ implies $\phi\in (\CH{m})^r$, we can define
$u_{\phi,j}:=\phi^{(j)}|_{\Z}$ for $j=0,\ldots,m$, i.e., $u_{\phi,j}(k):=\phi^{(j)}(k)$ for all $k\in \Z$. Then it follows from \cref{lem:phi:Z}
that \eqref{vgua:wj} holds with $m_a\ge m+1$. Moreover, by the same argument as in \cref{lem:phi:Z}, \eqref{phi:j} must hold, where $a_n:=2^{-n}\sd_a^n (\td I_r)$, i.e., $\wh{a_n}(\xi)=\wh{a}(2^{n-1}\xi)\cdots \wh{a}(2\xi)\wh{a}(\xi)$.

Using \eqref{vgua:wj} and our assumption in \eqref{bettermatch} for $u\in (\lp{0})^r$,
we have
\[
\wh{\vgu_a}(\xi)[\wh{u}(\xi)-\beta_j \wh{u_{\phi,j}}(\xi)]=
\wh{\vgu_a}(\xi)\wh{u}(\xi)-
\beta_j \wh{\vgu_a}(\xi)\wh{u_{\phi,j}}(\xi)
=\beta_j(i\xi)^j-\beta_j (i\xi)^j+\bo(|\xi|^{j+s})
=\bo(|\xi|^{j+s})
\]
as $\xi \to 0$, due to $j+s\le m+1\le m_a$ by \eqref{bettermatch}.
Therefore, $u-\beta_j u_{\phi,j}\in \PV_{j+s-1,\vgu_a}$.

By $m<\sm_\infty\le m+1$, we have $\rho_{m_a}(a,\vgu_a)_\infty=2^{-\sm_\infty(a)}\ge
2^{-m-1}$.
If $m_a=m+1$, then obviously $\rho_{m+1}(a,\vgu_a)_\infty=\rho_{m_a}(a,\vgu_a)_\infty$. If $m+1<m_a$,
then we deduce from \eqref{rhoj} and $\rho_{m_a}(a,\vgu_a)_\infty\ge 2^{-m-1}$ that $\rho_{m+1}(a,\vgu_a)_\infty=\max(2^{-m-1}, \rho_{m_a}(a,\vgu_a)_\infty)=\rho_{m_a}(a,\vgu_a)_\infty$.
Hence, we always have $\rho_{m+1}(a,\vgu_a)_\infty=\rho_{m_a}(a,\vgu_a)_\infty=2^{-\sm_\infty(a)}$.
If $s=m+1-j$, then we trivially have
$\rho_{j+s}(a,\vgu_a)_\infty
=\rho_{m+1}(a,\vgu_a)_\infty
=2^{-\sm_\infty(a)}$.
If $1\le s<m+1-j$, then by \eqref{rhoj}
we have
$\rho_{j+s}(a,\vgu_a)_\infty=\max(2^{-j-s}, \rho_{m_a}(a,\vgu_a)_\infty)=2^{-j-s}$ because $j+s\le m+1\le m_a$ and $\rho_{m_a}(a,\vgu_a)_\infty=2^{-\sm_\infty(a)}<2^{-m}$ by $\sm_\infty(a)>m$.
That is, we have $\rho_{j+s}(a,\vgu_a)_\infty=2^{-j-\nu}$ with $\nu:=\min(s,\sm_\infty(a)-j)>0$
for all $1\le s\le m+1-j$ and $j=0,\ldots,m$.
By $u-\beta_j u_{\phi,j}\in \PV_{j+s-1,\vgu_a}$ and the definition of $\rho_{j+s}(a,\vgu_a)_\infty$, for any $0<\varepsilon<\nu$, there exists a positive constant $C$ such that
\be \label{diff:u:wj}
2^{(j+1)n} \|a_n*(u-\beta_j u_{\phi,j})\|_{(\lp{\infty})^r}
\le C 2^{jn} 2^{-(j+\nu-\varepsilon)n}=
C 2^{-(\nu-\varepsilon)n},\qquad \forall\; n\in \N.
\ee
Now \eqref{phi:fastconverg} follows directly from the proved identity \eqref{phi:j} and the above inequality, because
\[
\| (a_n*u)(\cdot) 2^{(j+1)n}-\beta_j \phi^{(j)}(2^{-n}\cdot)\|_{(\lp{\infty})^r}
=2^{(j+1)n} \| (a_n*u)(\cdot)-(a_n*(\beta_j u_{\phi,j}))(\cdot)\|_{(\lp{\infty})^r}
\le C 2^{-(\nu-\varepsilon)n}.
\]
By the same proof of (2)$\imply$(3) in \cref{thm:main},
\eqref{vsd:fastconverg} follows directly from \eqref{phi:fastconverg}.
\end{proof}

\section{Strengthen Lagrange, Hermite and Generalized Hermite Subdivision Schemes}
\label{sec:hsd}

In this section we shall strength and generalize the current definitions of Lagrange, Hermite and generalized Hermite subdivision schemes and then discuss various types of vector subdivision schemes.
%We shall also provide a few examples to demonstrate our results on vector subdivision schemes.

Recall from \cref{sec:intro} that a (slightly generalized) convergent Lagrange subdivision scheme is

\begin{definition}\label{def:lsd}
\normalfont
Let $r\in \N$ and $m\in \NN$.
We say that \emph{the Lagrange subdivision scheme with mask $a\in \lrs{0}{r}{r}$ is convergent with limiting functions in $\CH{m}$} if for any input vector sequence $w_0\in \lrs{}{1}{r}$, there exist functions $\eta_1,\ldots,\eta_r\in \CH{m}$ such that for every constant $K>0$,
\be \label{lsd:converg:0}
\lim_{n\to \infty} \sup_{k\in \Z \cap [-2^n K, 2^n K]} |(\sd_a^n w_0)(k)e_\ell-\eta_\ell(2^{-n}k)|=0,
\qquad \ell=1,\ldots,r.
\ee
\end{definition}

We now recall the definition of a convergent Hermite subdivision scheme of order $r$, e.g., see \cite[Definition~1]{dm09},
\cite[Definition~1.1]{hyx05}, \cite[Definition~2]{ms19}, \cite[Definition~1]{han20} and references therein.

\begin{definition}\label{def:hsd}
Let $r\in \N$ and $m\in \NN$ with $m\ge r-1$.
We say that \emph{the Hermite subdivision scheme of order $r$
with mask $a\in \lrs{0}{r}{r}$
is convergent with limiting functions in $\CH{m}$} if for any input sequence $w_0\in \lrs{}{1}{r}$, there exists a scalar function $\eta\in \CH{m}$ such that for every constant $K>0$,
\be \label{hsd:converg:0}
\lim_{n\to \infty} \sup_{k\in \Z \cap [-2^n K, 2^n K]} |(\sd_a^n w_0)(k) 2^{(\ell-1)n}e_\ell-\eta^{(\ell-1)}(2^{-n}k)|=0,
\qquad
\ell=1,\ldots,r.
\ee
\end{definition}

The above convergent Lagrange and Hermite subdivision schemes are special cases of convergent generalized Hermite subdivision schemes of type $\ind$ in \cite[Definition~1]{han21} which is given below.

\begin{definition}\label{def:ghsd}
Let $m\in \NN$ and $\ind:=\{\nu_1,\ldots,\nu_r\}\subseteq \{0,\ldots,m\}$ be an ordered multiset with $\nu_1=0$.
We say that \emph{the generalized Hermite subdivision scheme of type $\ind$ with mask $a\in \lrs{0}{r}{r}$ is convergent with limiting functions in $\CH{m}$} if for any input vector sequence $w_0\in \lrs{}{1}{r}$, there exists a scalar function $\eta\in \CH{m}$ such that for every constant $K>0$,
\be \label{ghsd:converg:0}
\lim_{n\to \infty} \sup_{k\in \Z \cap [-2^n K, 2^n K]} |(\sd_a^n w_0)(k) 2^{\nu_\ell n}e_\ell-\eta^{(\nu_\ell)}(2^{-n}k)|=0,
\qquad
\ell=1,\ldots,r.
\ee
\end{definition}

Lagrange and Hermite subdivision schemes have been extensively studied in \cite{ccmm21,ch19,dm06,dm09,dl95,han01,hanbook,han03,han20,mer92,ms17,ms19,mhc20,rv20,zhou00} and references therein, while multivariate generalized Hermite subdivision schemes of type $\ind$ have been systematically studied in \cite{han21}.

We now demonstrate how our results can be applied to further improve known results on convergence of Lagrange, Hermite and generalized Hermite subdivision schemes in \cref{def:lsd,def:hsd,def:ghsd}.

\begin{theorem} \label{thm:lsd}
Let $r\in \N$, $m\in \NN$ and
$a\in \lrs{0}{r}{r}$ such that the condition in \eqref{cond:a:0} on $\wh{a}(0)$ is satisfied.
Let $\phi$ be a vector of compactly supported distributions satisfying $\wh{\phi}(2\xi)=\wh{a}(\xi)\wh{\phi}(\xi)$ and $\wh{\phi}(0)\ne 0$.
Then the Lagrange subdivision scheme with mask $a$ is convergent with limiting functions in $\CH{m}$ as in \cref{def:lsd} if and only if $\sm_\infty(a)>0$ and $\sm_\infty(\phi)>m$.
\end{theorem}

\begin{proof}
By \eqref{cond:a:0}, $1$ is a simple eigenvalue of $\wh{a}(0)$ and hence there exists a unique vector $\vec{v}\in \C^{1\times r}$ such that $\vec{v}\wh{a}(0)=\vec{v}$ and $\vec{v}\wh{\phi}(0)=1$. Define $\vgu_a:=\td \vec{v}$. Clearly, $\vgu_a$ is an order $1$ matching filter of the mask $a$ satisfying $\wh{\vgu_a}(0)\wh{a}(0)=\wh{\vgu_a}(0)$ and $\wh{\vgu_a}(0)\wh{\phi}(0)=1$.
By our assumption in \eqref{cond:a:0} on $\wh{a}(0)$, up to a multiplicative constant, $\phi$ is the unique refinable vector function satisfying $\wh{\phi}(2\xi)=\wh{a}(\xi)\wh{\phi}(\xi)$.

Sufficiency. By $\sm_\infty(a)>0$ and $\sm_\infty(\phi)>m$, we have $\phi\in (\CH{m})^r$ and we conclude from \cref{thm:main} with $m=0$ that \eqref{lsd:converg:0} holds with $\eta_\ell:=[\wh{\vgu_a}(0) e_\ell] w_0*\phi \in \CH{m}$ for $\ell=1,\ldots,r$.

Necessity. Suppose that \eqref{lsd:converg:0} holds.
Because $\wh{\vgu_a}(0)\ne 0$, there exists $\ell\in \{1,\ldots,r\}$ such that $\wh{\vgu_a}(0) e_\ell\ne 0$.
Define $u:=\td e_\ell$. By \eqref{lsd:converg:0} with $w_0=\td I_r$, \eqref{phi:deriv:u} must hold with $\tau=0$ for some vector function $\eta\in (\CH{m})^r$. Note that $\beta_0:=\wh{\vgu_a}(0)\wh{u}(0)=
\wh{\vgu_a}(0)e_\ell\ne 0$. Now
Applying item (3) of \cref{thm:mfilter} with $u:=\td e_\ell$ and $m=0$, we conclude that $\eta=\beta_0 \phi$. Since $\eta\in (\CH{m})^r$ and $\beta_0\ne 0$, we must have $\phi\in (\CH{m})^r$. Because $\phi$ is a compactly supported refinable vector function in $(\CH{m})^r$ with a finitely supported mask $a\in \lrs{0}{r}{r}$, by \cite[Corollary~5.8.2]{hanbook}, we must have $\sm_\infty(\phi)>m$.
Take $\PB:=\{e_1,\ldots,e_r\} \subseteq \PV_{-1,\vgu_a}$. Then trivially $\mspan\{u(\cdot-k) \setsp u\in \PB, k\in \Z\}=\lrs{0}{1}{r} \supseteq \PV_{0,\vgu_a}$.
Hence,
$\PB$ satisfies the condition in \cref{thm:main} with $m=0$.
By \eqref{lsd:converg:0},
item (4) of \cref{thm:main} with $m=0$ holds. Hence, we conclude from \cref{thm:main} with $m=0$ that $\sm_\infty(a)>0$.
\end{proof}

In \cref{thm:lsd}, if $\sm_\infty(a)>m$ with $m\in \NN$, then we trivially have $\sm_\infty(a)>m\ge 0$ and $\sm_\infty(\phi)\ge \sm_\infty(a)>m$. Therefore, $\sm_\infty(a)>m$ is a sufficient condition for the convergence of the Lagrange subdivision scheme with mask $a$ and limiting functions in $\CH{m}$ in \cref{def:lsd}.

If $\vgu_a\in \lrs{0}{1}{r}$ is an order $m+1$ matching filter of a matrix mask $a\in \lrs{0}{r}{r}$, then it is trivial to observe from its definition in \eqref{mfilter} that $\beta \vgu_a$ is also an order $m+1$ matching filter of $a$ for any $\beta\in \C\bs\{0\}$. To normalize a matching filter $\vgu_a\in \lrs{0}{1}{r}$, since $\wh{\vgu_a}(0)\ne 0$, from now on we can always assume that the first nonzero entry of $\wh{\vgu_a}(0)$ is $1$.

For the convergence of Hermite subdivision schemes in \cref{def:hsd}, we have the following result, which improves \cite[Theorems~1.2 and 1.3]{han20} by dropping some conditions on $\phi$.

\begin{theorem} \label{thm:hsd}
Let $r\in \N$ and $m\in \NN$ with $m\ge r-1$.
Let $a\in \lrs{0}{r}{r}$ satisfying \eqref{cond:a:0} and
let $\vgu_a\in \lrs{0}{1}{r}$ be an order $r$ matching filter of $a$ such that the first nonzero entry of $\wh{\vgu_a}(0)$ is $1$.
Let $\phi$ be a vector of compactly supported distributions satisfying $\wh{\phi}(2\xi)=\wh{a}(\xi)\wh{\phi}(\xi)$ and $\wh{\vgu_a}(0)\wh{\phi}(0)=1$.
\begin{enumerate}
\item[(1)] If the Hermite subdivision scheme with mask $a$ is convergent with limiting functions in $\CH{m}$ as in \cref{def:hsd}, then $\sm_\infty(a)>0$, $\sm_\infty(\phi)>m$, and the matching filter $\vgu_a$ must satisfy
\be \label{vgua:hermite:r}
\wh{\vgu_a}(\xi)=[1+\bo(|\xi|), i\xi+\bo(|\xi|^2),\ldots, (i\xi)^{r-1}+\bo(|\xi|^{r})],\quad \xi \to 0.
\ee

\item[(2)] If $\sm_\infty(a)>r-1$, $\sm_\infty(\phi)>m$ and the matching filter $\vgu_a$ satisfies \eqref{vgua:hermite:r},
then the Hermite subdivision scheme with mask $a$ is convergent with limiting functions in $\CH{m}$.
\end{enumerate}
\end{theorem}

\begin{proof}
We first prove item (1). By \eqref{hsd:converg:0} with the initial sequence $w_0=\td I_r$ of matrices (instead of vectors), there must exist a vector function $\eta\in (\CH{m})^r$ such that
\be \label{hsd:converg:2}
\lim_{n\to \infty} \| (\sd_a^n (\td I_r))(\cdot) 2^{(\ell-1)n}e_\ell -\eta^{(\ell-1)}(2^{-n}\cdot)\|_{(\lp{\infty})^r}=0,\qquad
\ell=1,\ldots,r.
\ee
We first prove that $\wh{\vgu_a}(0)=[1,0,\ldots,0]$.
By \eqref{hsd:converg:2}, noting that
$\sd_a^n(\td I_r) e_\ell=(\sd_a^n(\td I_r))*(\td e_\ell)$, we trivially conclude from \eqref{hsd:converg:2} and $\ell-1\ge  1$ for all $\ell=2,\ldots,r$ that
\be \label{sda:eta}
\lim_{n\to \infty} \| (\sd_a^n (\td I_r))* (\td e_\ell)\|_{(\lp{\infty})^r}=0, \qquad \ell=2,\ldots,r.
\ee
That is, \eqref{phi:deriv:u} is satisfied with $\tau=0$, $\eta=0$ and $u=\td e_\ell$ for $\ell=2,\ldots,r$.
If $\wh{\vgu_a}(0) e_\ell \ne 0$ for some $\ell=2,\ldots,r$, then $u:=\td e_\ell$ satisfies $\wh{u}(0)=e_\ell$ and $\wh{\vgu_a}(0)\wh{u}(0)\ne 0$; Hence, it follows from \eqref{sda:eta} and item (3) of \cref{thm:mfilter} with $m=0$ that $\eta$ in \eqref{phi:deriv:u} cannot be identically zero, a contradiction to $\eta=0$ in \eqref{sda:eta}. This proves $\wh{\vgu_a}(0) e_\ell=0$ for all $\ell=2,\ldots,r$. Because the first nonzero entry of $\wh{\vgu_a}(0)$ is $1$, we must have $\wh{\vgu_a}(0)=[1,0,\ldots,0]$. Define $u:=\td e_1$. Because $\wh{\vgu_a}(0) \wh{u}(0)=\wh{\vgu_a}(0) e_1=1 \ne 0$, by \eqref{hsd:converg:2} and item (3) of \cref{thm:mfilter} with $m=0$, we conclude that
$\eta$ cannot be identically zero and $\eta=\beta_0 \varphi$, where $\beta_0:=\wh{\vgu_a}(0) e_1=1$ and
$\varphi$ is a compactly supported continuous refinable vector function satisfying $\wh{\varphi}(2\xi)=\wh{a}(\xi)\wh{\varphi}(\xi)$ and $\wh{\vgu_a}(0)\wh{\varphi}(0)=1$. Due to the condition in \eqref{cond:a:0}, such a refinable vector function $\varphi$ must be unique and we must have $\varphi=\phi$. Hence, $\eta=\beta_0 \phi=\phi$ by $\beta_0=1$. Because $\eta\in (\CH{m})^r$, we conclude that $\phi\in (\CH{m})^r$.
Because $\phi$ is a compactly supported refinable vector function with a finitely supported mask $a\in \lrs{0}{r}{r}$, by \cite[Corollary~5.8.2]{hanbook}, we conclude from $\phi\in (\CH{m})^r$ that $\sm_\infty(\phi)>m$.

We now prove \eqref{vgua:hermite:r}. Because $\eta=\phi$ and $\phi\in (\CH{m})^r$ has compact support, the vector function $\phi^{(j)}$ cannot be identically zero for all $j=0,\ldots,m$. For $\ell=1,\ldots,r$, define $u_\ell:=\td e_\ell$. From \eqref{sda:eta} with $\eta=\phi$, we conclude from \cref{thm:mfilter} that
$\ell-1=\ld_m(\wh{\vgu_a}\wh{u_\ell})$ and $\beta_{\ell-1}=\frac{[\wh{\vgu_a}\wh{u_\ell}]^{(\ell-1)}(0)}{i^{\ell-1} (\ell-1)!}=1$. Therefore,
\[
\wh{\vgu_a}(\xi) e_{\ell}=\wh{\vgu_a}(\xi)\wh{u_\ell}(\xi)=
\beta_{\ell-1}(i\xi)^{\ell-1}+\bo(|\xi|^\ell)
=(i\xi)^{\ell-1}+\bo(|\xi|^\ell),\qquad \xi\to 0
\]
for all $\ell=1,\ldots,r$. This proves \eqref{vgua:hermite:r}.

Take $\PB:=\{e_1,\ldots,e_r\} \subseteq \PV_{-1,\vgu_a}$. Then trivially $\mspan\{u(\cdot-k) \setsp u\in \PB, k\in \Z\}=\lrs{0}{1}{r} \supseteq \PV_{0,\vgu_a}$.
Hence,
$\PB$ satisfies the condition in \cref{thm:main} with $m=0$.
Because we proved $\wh{\vgu_a}(0)=[1,0,\ldots, 0]=e_1^\tp$, we deduce from \eqref{hsd:converg:2} that
item (5) of \cref{thm:main} with $m=0$ holds. Hence, we conclude from \cref{thm:main} with $m=0$ that $\sm_\infty(a)>0$.
This proves item (1).

To prove item (2), by \cref{thm:main} with $m=r-1$ and \eqref{vgua:hermite:r}, we conclude that \eqref{hsd:converg:2} must hold with $\eta=\phi$.
Since $\sm_\infty(\phi)>m$ implies $\phi\in (\CH{m})^r$, the claim in \cref{def:hsd} follows directly from \eqref{hsd:converg:2} with $\eta=\phi\in (\CH{m})^r$.
\end{proof}

Under the assumption that the basis vector function $\phi$ in a convergent Hermite subdivision scheme in \cref{def:hsd} satisfies
\[
\mspan\{\wh{\phi}(2\pi k) \setsp k\in \Z\}=\C^r
\quad \mbox{and}\quad \mspan\{\wh{\phi}(\pi+2\pi k) \setsp k\in \Z\}=\C^r,
\]
then \cite[Theorem~1.2]{han20} shows that
\eqref{cond:a} must hold and an order $m+1$ matching filter $\vgu_a\in \lrs{0}{1}{r}$ determined by \eqref{vgua:value} must satisfy \eqref{vgua:hermite:r}.
Consequently, we proved the condition \eqref{vgua:hermite:r} on $\vgu_a$ in \cref{thm:hsd} under a much weaker condition of an order $m+1$ matching filter $\vgu_a\in \lrs{0}{1}{r}$ of the matrix mask $a$.
Moreover, under the condition that
\begin{enumerate}
\item[(i)] a matrix mask $a\in \lrs{0}{r}{r}$ has order $m+1$ sum rules with an order $m+1$ matching filter $\vgu_a\in \lrs{0}{1}{r}$ of $a$ satisfying \eqref{vgua:hermite:r},
\end{enumerate}
\cite[Theorem~1.3]{han20} shows that $\sm_\infty(a)>m$ guarantees the convergence of the Hermite subdivision scheme with mask $a$ and limiting functions in $\CH{m}$. Conversely, under the condition in the above item (i) and
the additional condition:
\begin{enumerate}
\item[(ii)] the integer shifts of the refinable vector function $\phi$ with the matrix mask $a$ are stable, i.e., $\mspan\{\wh{\phi}(\xi+2\pi k) \setsp k\in \Z\}=\C^r$ for all $\xi\in \R$,
\end{enumerate}
then \cite[Theorem~1.3]{han20} proves that if the Hermite subdivision scheme with mask $a$ is convergent with limiting functions in $\CH{m}$, then $\sm_\infty(a)>m$. The stability condition on $\phi$ guarantees $\sm_\infty(a)=\sm_\infty(\phi)$ and hence $\sm_\infty(\phi)>m$ in \cref{thm:lsd,thm:hsd} implies $\sm_\infty(a)>m$, which is just item (1) of \cref{thm:main}. Therefore, \cref{thm:hsd} improves \cite[Theorems~1.2 and~1.3]{han20}.
Without the stability condition on $\phi$,
it is well known  that $\sm_\infty(\phi)>\sm_\infty(a)$ often happens and some examples are given in \cite[Section~5]{han20} for Hermite subdivision schemes in \cref{def:hsd}.
See \cite{han06} and references therein about how to characterize $\sm_p(\phi)$ without the stability condition on $\phi$ as stated in item (ii).
Therefore, according to \cref{thm:lsd,thm:hsd},
the basis vector functions $\phi$ in \cref{def:lsd,def:hsd} could satisfy $\sm_\infty(\phi)>m$ for a very large integer $m$, however, $\sm_\infty(a)$ could be very small, say $\sm_\infty(a)<m_0<m$. Despite the high smoothness of the refinable vector function $\phi\in (\CH{m})^r$, we cannot employ the vector subdivision scheme in \cref{def:vsd} to compute $\phi^{(j)}$ for $m_0<j\le m$ at all. In this sense, \cref{def:lsd,def:hsd,def:ghsd} are weaker than the vector subdivision scheme defined in \cref{def:vsd}.

By the exact same proof of \cref{thm:hsd}, for convergent generalized Hermite subdivision schemes, we have the following result, which improves \cite[Theorems~3.1 and~4.3]{han21} for dimension one and includes \cref{thm:lsd,thm:hsd} as special cases.

\begin{theorem} \label{thm:ghsd}
Let $m\in \NN$ and $\ind:=\{\nu_1,\ldots,\nu_r\}\subseteq \{0,\ldots,m\}$ be an ordered multiset with $\nu_1=0$. Define $m_\ind:=\max(\nu_1,\ldots,\nu_r)$.
Let $a\in \lrs{0}{r}{r}$ such that $a$ satisfies \eqref{cond:a:0}.
Let $\vgu_a\in \lrs{0}{1}{r}$ be an order $m_\ind+1$ matching filter of the mask $a$ such that the first nonzero entry of $\wh{\vgu_a}(0)$ is $1$.
Let $\phi$ be a vector of compactly supported distributions satisfying $\wh{\phi}(2\xi)=\wh{a}(\xi)\wh{\phi}(\xi)$ and $\wh{\vgu_a}(0)\wh{\phi}(0)=1$.
\begin{enumerate}
\item[(1)] If the generalized Hermite subdivision scheme of type $\ind$ with mask $a$ is convergent with limiting functions in $\CH{m}$ as in \cref{def:ghsd}, then $\sm_\infty(a)>0$, $\sm_\infty(\phi)>m$, and the matching filter $\vgu_a$ of the matrix mask $a$ must satisfy
\be \label{vgua:ghermite:r}
\wh{\vgu_a}(\xi)=[(i\xi)^{\nu_1}+\bo(|\xi|^{\nu_1+1}), (i\xi)^{\nu_2}+\bo(|\xi|^{\nu_2+2}),\ldots, (i\xi)^{\nu_r}+\bo(|\xi|^{\nu_r+1})],\quad \xi \to 0.
\ee

\item[(2)] If $\sm_\infty(a)>m_\ind$, $\sm_\infty(\phi)>m$ and the matching filter $\vgu_a$ satisfies \eqref{vgua:ghermite:r},
then the generalized Hermite subdivision scheme of type $\ind$ with mask $a$ is convergent with limiting functions in $\CH{m}$ as in \cref{def:ghsd}.
\end{enumerate}
\end{theorem}

The above results and discussions motivate us to strengthen \cref{def:lsd,def:hsd,def:ghsd} for Lagrange, Hermite and generalized Hermite subdivision schemes and to introduce new types of vector subdivision schemes.
%Let $r\in \N$ and $m\in \NN$.
Let $a\in \lrs{0}{r}{r}$ such that $a$ satisfies \eqref{cond:a:0}.
Let $\vgu_a\in \lrs{0}{1}{r}$ be an order $m+1$ matching filter of $a$ such that the first nonzero entry of $\wh{\vgu_a}(0)$ is $1$.
Let $\phi=[\phi_1,\ldots,\phi_r]^\tp$ be the unique vector of compactly supported distributions satisfying $\wh{\phi}(2\xi)=\wh{a}(\xi)\wh{\phi}(\xi)$ and $\wh{\vgu_a}(0)\wh{\phi}(0)=1$.
By \cref{thm:main,thm:fastconverg},
we say that the vector subdivision scheme with mask $a\in \lrs{0}{r}{r}$ in \cref{def:vsd} is

\begin{enumerate}
\item[(1)] a $\CH{m}$ convergent Lagrange subdivision scheme with mask $a$ if $\sm_\infty(a)>m$ and $\wh{\vgu_a}(0)=[1,\ldots,1]$. Then by \cref{thm:main}, $\phi\in (\CH{m})^r$ and \eqref{lsd:converg} must hold with $\eta=w_0*\phi\in \CH{m}$. Moreover, we say that
    a $\CH{m}$ convergent Lagrange subdivision scheme is fast convergent if $\sm_\infty(a)>m$ and $\wh{\vgu_a}(\xi)=[1,\ldots,1]+\bo(|\xi|^{m+1})$
    as $\xi \to 0$.
Then by \cref{thm:fastconverg}, for any $0<\varepsilon<\nu$ with $\nu:=\min(m+1,\sm_\infty(a))$, there exists a positive constant $C$ such that
\be \label{lsd:fastconverg}
\| (\sd_a^n(\td I_r))(\cdot)e_\ell-\phi(2^{-n}\cdot)\|_{(\lp{\infty})^r}
\le C 2^{-(\nu-\varepsilon)n}, \qquad \forall\, n\in \N, \ell=1,\ldots,r.
\ee
\item[(2)] a $\CH{m}$ convergent Hermite subdivision scheme of order $r$ with mask $a$ if $\sm_\infty(a)>m\ge r-1$ and the order $m+1$ matching filter $\vgu_a$ of $a$ satisfies \eqref{vgua:hermite:r}.
Then by \cref{thm:main}, \eqref{hsd:converg:0} holds with
$\eta=w_0*\phi\in \CH{m}$.
Moreover, we say that a $\CH{m}$ convergent Hermite subdivision scheme of order $r$ is fast convergent if $\sm_\infty(a)>m$ and
\be  \label{vgua:hemrite}
\wh{\vgu_a}(\xi)=
[1,i\xi,\ldots,(i\xi)^{r-1}]+\bo(|\xi|^{m+1}),\quad \xi \to 0.
\ee
Then by \cref{thm:fastconverg}, for $j=0,\ldots, m$ and $0<\varepsilon<\nu$ with $\nu:=\min(m-r+2,\sm_\infty(a)-r+1)$, there exists $C>0$ such that
\[
\| (\sd_a^n(\td I_r))(\cdot)e_j 2^{(j-1)n}-\phi^{(j-1)}(2^{-n}\cdot)\|_{(\lp{\infty})^r}
\le C 2^{-(\sm_\infty(a)-j+1-\varepsilon)n}, \quad \forall\, n\in \N, j=1,\ldots,r.
\]

\item[(3)] a $\CH{m}$ convergent scalar-type vector subdivision scheme with mask $a$ if $\sm_\infty(a)>m$ and
$\wh{\vgu_a}(\xi)=[1+\bo(|\xi|), \bo(|\xi|^{m+1}),\ldots, \bo(|\xi|)^{m+1}]$ as $\xi \to 0$.
Then by \cref{thm:main}, we must have
$\lim_{n\to \infty}
\| (\sd_a^n(\td I_r))(\cdot)e_1-\phi(2^{-n}\cdot)\|_{(\lp{\infty})^r}
=0$ and
$\lim_{n\to \infty}
2^{mn}\|(\sd_a^n(\td I_r))(\cdot)e_\ell\|_{(\lp{\infty})^r}
=0$ for all $\ell=2,\ldots,r$.
Moreover, we say that a $\CH{m}$ convergent scalar-type vector subdivision scheme is fast convergent if $\sm_\infty(a)>m$ and $\wh{\vgu_a}(\xi)=[1,0,\ldots,0] + \bo(|\xi|^{m+1})$ as $\xi \to 0$.
Then for any $0<\varepsilon<\nu$ with $\nu:=\min(m+1,\sm_\infty(a))$, \eqref{lsd:fastconverg} with $\ell=1$ holds.

\item[(4)] a $\CH{m}$ convergent generalized Hermite subdivision scheme of type $\ind$ if $\sm_\infty(a)>m$ and
\be  \label{vgua:birkhoff}
\wh{\vgu_a}(\xi)=
[1+\bo(|\xi|),(i\xi)^{\nu_2}+\bo(|\xi|^{\nu_2+1}),
\ldots, (i\xi)^{\nu_r}+\bo(|\xi|^{\nu_r+1})]+\bo(|\xi|^{m+1}),\quad \xi \to 0,
\ee
where $\ind:=\{\nu_1,\ldots, \nu_r\}\subseteq \{0,\ldots,m\}$ with $\nu_1=0$.
Then by \cref{thm:main},
\[
\lim_{n\to \infty}
\| (\sd_a^n(\td I_r))(\cdot)e_j 2^{\nu_j n}-
\phi^{(\nu_j)}(2^{-n}\cdot)\|_{(\lp{\infty})^r}
=0,\qquad \forall\; j=1,\ldots,r.
\]
\item[(5)] a $\CH{m}$ convergent balanced vector subdivision scheme if $\sm_\infty(a)>m$ and
\be \label{vgua:bsd}
\wh{\vgu_a}(\xi)=\wh{c}(\xi)[1, e^{i\xi/r},\ldots, e^{i\xi(r-1)/r}]+\bo(|\xi|^{m+1}),\quad \xi \to 0 \quad \mbox{with}\quad
c\in \lp{0}, \wh{c}(0)=1.
\ee
Balanced vector subdivision schemes enjoy the property $\PR_{m,\vgu_a}:=\{ \pq*\vgu_a \setsp \pq\in \PL_m\}=\{ (\pq(r\cdot), \pq(r\cdot+1),\ldots, \pq(r\cdot+r-1)) \setsp \pq\in \PL_m\}$ and play the key role in the balanced fast multiwavelet transform, e.g., see \cite{han09,hanbook} and references therein. If in addition $c=\td$, then the vector subdivision scheme has order $m+1$ linear-phase moments for the polynomial-interpolation property, see \cite[Theorem~5.3]{han21} for details.
\end{enumerate}

%Obviously, all the vector subdivision schemes in items (1)--(3) are special cases of the generalized Hermite subdivision schemes in item (4).
Given many types of vector subdivision schemes in the above discussion, a natural question is whether they are essentially different to each other. In particular, whether a strengthened Lagrange subdivision scheme in item (1) is really different to the strengthened Hermite subdivision schemes in item (3).
To do so, we need to recall the notion of the normal form of a matrix mask.
For $U\in \lrs{0}{r}{r}$, we say that $U$ (or $\wh{U}$) is \emph{strongly invertible} if $\det(\wh{U})$ is a nonzero monomial, in other words, $(\wh{U}(\xi))^{-1}$ is a matrix of $2\pi$-periodic trigonometric polynomials.
Using the normal form of matrix masks in \cite[Theorem~5.6.4]{hanbook} and \cite[Theorem~4.1]{han20}, we have the following result, which generalizes \cite[Theorem~1.4]{han20}.

\begin{prop}\label{prop:vsd}
Let $r\in \N$ and $m\in \NN$.
Let $\mathring{a}\in \lrs{0}{r}{r}$ such that $\sm_\infty(\mathring{a})>m$ (such a matrix mask $\mathring{a}$ always exists). If $r>1$, then
for any given $\vgu_a\in \lrs{0}{1}{r}$ with $\wh{\vgu_a}(0)\ne 0$,
there exists a strongly invertible sequence $U\in \lrs{0}{r}{r}$ such that $\sm_\infty(a)=\sm_\infty(\mathring{a})>m$ and $\vgu_a$ is an order $m+1$ matching filter of $a$, where $a\in \lrs{0}{r}{r}$ is defined by $\wh{a}(\xi):=\wh{U}(2\xi)\wh{\mathring{a}}(\xi) (\wh{U}(\xi))^{-1}$.
\end{prop}

\bp The existence of a desired matrix mask $\mathring{a}$ is guaranteed by \cite[Proposition~6.2]{han09} or \cite[Theorem~1.4]{han20}.
Because $\sm_\infty(\mathring{a})>m$, by \cref{thm:vca}, \eqref{cond:a} must holds for $\mathring{a}$ and hence $\mathring{a}$ must have an order $m+1$ matching filter $\mathring{\vgu}\in \lrs{0}{1}{r}$ of the mask $\mathring{a}$.
By \cite[Theorem~2.1]{han09},
there exist strongly invertible sequences $U_1, U_2\in \lrs{0}{r}{r}$ such that
\be \label{nf:U12}
\wh{\vgu_a}(\xi)\wh{U_1}(\xi)=[1,0,\ldots,0]+\bo(|\xi|^{m+1})
\quad \mbox{and}\quad
\wh{\mathring{\vgu}}(\xi)\wh{U_2}(\xi)=[1,0,\ldots,0]+\bo(|\xi|^{m+1}),\quad \xi \to0.
\ee
For the convenience of the reader,
we provide a self-contained proof of \eqref{nf:U12} using the idea in \cite{han03,han09} for the normal form of matrix masks.
Write $[\vgu_1,\ldots,\vgu_r]:=\vgu_a$.
Because $\wh{\vgu_a}(0)\ne 0$, without loss of generality we can assume $\wh{\vgu_1}(0)\ne 0$, otherwise we perform a permutation on the entries of $\vgu_a$. Since $\wh{\vgu_1}(0)\ne 0$, we can easily find $u_1,\ldots,u_r\in \lp{0}$ such that
\[
\wh{u_1}(\xi)=1/\wh{\vgu_1}(\xi)+\bo(|\xi|^{m+1})
\quad \mbox{and}\quad
\wh{u_\ell}(\xi)=\wh{\vgu_\ell}(\xi)/\wh{\vgu_1}(\xi)+\bo(|\xi|^{m+1}),\qquad \xi \to 0, \ell=2,\ldots,r.
\]
Using the binomial expansion and $\wh{u_1}(0)\ne 0$, we can write
$(1-\wh{u_1}(\xi)/\wh{u_1}(0))^{m+1}=
1-\wh{u_1}(\xi)\wh{g}(\xi)$ for a unique $2\pi$-periodic trigonometric polynomial $\wh{g}$.
Define $U_1\in \lrs{0}{r}{r}$ by
\[
\wh{U_1}(\xi):=\frac{1}{\wh{\vgu_1}(0)}\left[ \begin{matrix}
\wh{u_1}(\xi)+\wh{u_2}(\xi) &1-(\wh{u_1}(\xi)+\wh{u_2}(\xi))\wh{g}(\xi)
&-\wh{u_3}(\xi) &\cdots &-\wh{u_r}(\xi)\\
-1 &\wh{g}(\xi) &0 &\cdots &0\\
0 &0 &1 &\cdots &0\\
\vdots &\vdots &\vdots &\ddots &\vdots\\
0 &0 &0 &\cdots &1\end{matrix}\right].
\]
Note that $\det(\wh{U_1}(\xi))=1$ and the first identity in \eqref{nf:U12} holds. By the same argument, there exists a strongly invertible sequence $U_2\in \lrs{0}{r}{r}$ such that the second identity in \eqref{nf:U12} holds.
Define $\wh{U}(\xi):=\wh{U_2}(\xi)(\wh{U_1}(\xi))^{-1}$. Then $U$ is strongly invertible. Now it is straightforward to check
that $\vgu_a$ is an order $m+1$ matching filter of the matrix mask $a$ and $\sm_\infty(a)=\sm_\infty(\mathring{a})>m$.
\end{proof}

Finally, we discuss how to choose a desired $\PB$ in \cref{thm:main}.
Let $U_1=[u_1,\ldots, u_r]$ be constructed in the proof of \cref{prop:vsd} such that the first identity in \eqref{nf:U12} holds, where
$u_1,\ldots, u_r$ are the column vectors of $U_1$.
One possible choice of $\PB\subseteq \PV_{m-1,\vgu_a}$ in \cref{thm:main} is to choose $\PB=\{\nabla^{m} \td e_1, u_2,\ldots, u_r\}$.

\section{Some Examples of Various Types of Vector Subdivision Schemes}
\label{sec:ex}

In this section we give some examples of vector subdivision schemes whose matrix masks have short supports and symmetry. As an application of \cref{thm:fastconverg}, their convergence rates are also provided.

Though all different types of vector subdivision schemes can be constructed theoretically through \cref{prop:vsd}, the support of the matrix mask $a$ in \cref{prop:vsd} is often very large and may lack symmetry. Therefore, it is not that useful to construct vector subdivision schemes by \cref{prop:vsd} for practical purposes. We  now discuss how to construct particular vector subdivision schemes with short support, symmetry, and high smoothness. For a given integer $m_a\in \N$, to obtain all filters $a\in \lrs{0}{r}{r}$ such that $a$ has order $m_a$ sum rules with a matching filter $\vgu_a\in \lrs{0}{1}{r}$, it suffices for us to solve a system of equations induced by \eqref{sr} with $m=m_a-1$, explicitly,
\be \label{sr:ma}
\wh{\vgu_a}(2\xi)\wh{a}(\xi)=\wh{\vgu_a}(\xi)+\bo(|\xi|^{m_a})
\quad \mbox{and}\quad
\wh{\vgu_a}(2\xi)\wh{a}(\xi+\pi)=\bo(|\xi|^{m_a}), \qquad \xi \to 0.
\ee
Then we estimate $\sm_\infty(a)$ to check the condition $\sm_\infty(a)>m$.
By $\fs(a)$ we denote the smallest interval such that $a$ vanishes outside it.
In the following, we always define $\phi$ to be the unique refinable vector function satisfying $\wh{\phi}(2\xi)=\wh{a}(\xi)\wh{\phi}(\xi)$ and $\wh{\vgu_a}(0)\wh{\phi}(0)=1$.
Let $m\in \NN$ be the largest nonnegative integer such that $m<\sm_\infty(a)\le m+1$.
To measure the errors and convergence rates of the vector subdivision scheme with mask $a$, for $u\in (\lp{0})^r$ and $n\in \N$, we define
\be \label{En}
\err_u(n):=\| [(\sd_a^n(\td I_r))*u](\cdot) 2^{jn}-\beta_j \phi^{(j)}(2^{-n}\cdot)\|_{(\lp{\infty})^r}
\quad \mbox{with}\quad
j:=\ld_{m}(\wh{\vgu_a}\wh{u}),\;
\beta_j:=\frac{[\wh{\vgu_a} \wh{u}]^{(j)}(0)}{i^j j!}.
\ee
Then the convergence rate of the vector subdivision scheme is approximately given by the slope of $-\log_2 E_u(n)$ with the subdivision level $n\in \N$ as $n\to \infty$. We now present a few examples of several vector subdivision schemes with various properties.

\begin{example}\label{ex1}
\normalfont
Solving the equations induced by \eqref{sr:ma} with $m_a:=4$ and $\wh{\vgu_a}(0)=[1,1]$,
all symmetric masks $a\in \lrs{0}{2}{2}$
satisfying $a(-k)=a(k)$ for all $k\in \Z$
with $\fs(a)\subseteq [-1,1]$ are given by
\[
a=\left\{
\begin{bmatrix}
\tfrac{1}{16}+t_1-t_2 &t_1-t_2\\[0.1em]
\tfrac{3}{16}-t_1+t_2 &\tfrac{1}{4}-t_1+t_2
\end{bmatrix},
\begin{bmatrix}
\frac{1}{8}+2t_1 &2t_1\\[0.1em]
\frac{3}{8}-2t_1 &\frac{1}{2}-2t_1
\end{bmatrix},
\begin{bmatrix}
\tfrac{1}{16}+t_1-t_2 &t_1-t_2\\[0.1em]
\tfrac{3}{16}-t_1+t_2 &\tfrac{1}{4}-t_1+t_2
\end{bmatrix}
\right\}_{[-1,1]}
\]
with an order $4$ matching filter $\vgu_a\in \lrs{0}{1}{2}$ satisfying
\[
\wh{\vgu_a}(\xi)=\left[1+\tfrac{16t_1-3}{96 t_2} \xi^2, 1+\tfrac{t_1}{6t_2} \xi^2\right]+\bo(|\xi|^4),\qquad \xi \to 0,
\]
where $t_1, t_2\in \R$ and $t_2\ne 0$.
By calculation, $\sr(a)=4$, $\sm_2(a)=3/2$, and its unique refinable vector function $\phi=[\phi_1,\phi_2]^\tp$ has symmetry $\phi(-\cdot)=\phi$ and
is a vector spline with support $[-1,1]$ given by
{\footnotesize{\[
\phi(x)|_{[0,1]}=
\begin{bmatrix}
\frac{16}{3}(1-x)(t_2x^2 - 2t_2x + t_1)\\
\frac{1}{3}(x - 1)(16t_2x^2 - 32 t_2x
 + 16t_1 - 3)\end{bmatrix},
\quad
\phi(x)|_{[-1,0]}=
\begin{bmatrix}
\frac{16}{3}(1+x)(t_2x^2 + 2t_2x + t_1)\\
-\frac{1}{3}(x+1)(16t_2x^2 + 32 t_2x
 + 16t_1 - 3)\end{bmatrix}.
\]}}
Hence, $\sm_\infty(\phi)=1$. Also note that $\phi_1(x)+\phi_2(x)=\max(0,1-|x|)$, which is the hat function (i.e., the B-spline function of order $2$) supported on $[-1,1]$.
Because $\tz_{a}$ has the special eigenvalues $2^0, 2^{-1}, 2^{-2}, 2^{-3}$ plus the additional eigenvalues $2^{-1}, 2^{-3}$, we obtain $\sm_\infty(a)=1$ because $1=\sm_2(a)-1/2\le \sm_\infty(a)\le 1$.
Hence, by \cref{thm:main} or \cref{thm:lsd}, the vector subdivision scheme with mask $a$ is a $\CH{0}$ convergent Lagrange subdivision scheme. This example of Lagrange subdivision schemes is of particular interest in numerical PDEs, because the spline refinable vector function $\phi$ belongs to the Sobolev space $H^1(\R)$, has very short support $[-1,1]$ for building multiwavelets on the interval $[0,1]$ (e.g., see \cite{hm21} for adapting multiwavelets from the real line to bounded intervals), and have the approximation order $4$ for fast convergence rates of numerical schemes.
\end{example}

\begin{example}\label{ex2}
Solving the equations induced by \eqref{sr:ma} with $m_a:=8$ and $\wh{\vgu_a}(0)=[1,1]$,
all symmetric masks $a\in \lrs{0}{2}{2}$
satisfying $a(-k)=a(k)$ for all $k\in \Z$
with $\fs(a)\subseteq [-2,2]$ are given by two families. The first family is given by
{\footnotesize{\[
a_1(0)=\begin{bmatrix}
6t_1+4t_2 &-\frac{3}{128}+6t_1+4t_2\\[0.1em]
\frac{3}{8}-6t_1-4t_2 &\frac{51}{128}-6t_1-4t_2
\end{bmatrix},\quad
a_1(1)=\begin{bmatrix}
4t_1-t_2 &-\frac{1}{64}+4t_1-t_2\\[0.1em]
\frac{1}{4}-4t_1+t_2 &\frac{17}{64}-4t_1+t_2
\end{bmatrix},\quad
a_1(2)=\begin{bmatrix}
t_1 &-\frac{1}{256}+t_1\\[0.1em]
\frac{1}{16}-t_1 &\frac{17}{256}-t_1
\end{bmatrix}
\]}}
with an order $8$ matching filter $\vgu_{a_1}\in \lrs{0}{1}{2}$ satisfying
\begin{align*}
\wh{\vgu_{a_1}}(\xi)=&\left[1+\tfrac{1}{6}\xi^2
+\tfrac{4032t_2 + 2688t_1 - 168}{241920t_2}\xi^4+
\tfrac{384t_2 + 448t_1 - 28}{241920t_2}\xi^6,\right. \\
&\qquad\qquad \left. 1+\tfrac{1}{6} \xi^2+\tfrac{16128t_2 + 10752t_1 - 42}{967680t_2}\xi^4
+\tfrac{1536t_2 + 1792t_1 - 7}{967680t_2}\xi^6\right]+\bo(|\xi|^8),\qquad \xi\to 0,
\end{align*}
%
%as $\xi \to 0$,
where $t_1,t_2\in \R$ with $t_2\ne 0$. By calculation, we obtain $\sr(a_1)=8$, $\sm_2(a_1)=7/2$, and
its unique refinable vector function $\phi=[\phi_1,\phi_2]^\tp$ has symmetry $\phi(-\cdot)=\phi$  and
is a vector spline with support $[-2,2]$ given by
{\footnotesize{
\begin{align*}
&\phi(x)|_{[0,1]}=
\begin{bmatrix}
-\frac{32}{35} t_2 x^7+\frac{64}{15} t_2 x^6
-\frac{128}{3}t_2 x^4+(\frac{128}{15}t_1+ \frac{256}{3}t_2 - \frac{1}{30})x^3+
-(\frac{256}{15}t_1+ \frac{256}{5}t_2- \frac{1}{15})x^2+
(\frac{512}{45}t_1+\frac{512}{105}t_2-\frac{2}{45})\\[0.1em]
\frac{1}{630}(x - 2)^3(192t_2x^4 - 1536t_2x^3 +4608t_2x^2 - 6144t_2x - 1792t_1 + 3072t_2 + 7)
\end{bmatrix},\\
&\phi(x)|_{[1,2]}=
\begin{bmatrix}
\frac{32}{35}t_2x^7 - \frac{64}{15}t_2x^6 + \frac{128}{3}t_2x^4 -(\frac{128}{15}t_1 + \frac{256}{3}t_2 -\frac{8}{15})x^3 + (\frac{256}{15}t_1 + \frac{256}{5}t_2 - \frac{16}{15})x^2 - \frac{512}{45}t_1 - \frac{512}{105}t_2 + \frac{32}{45}\\[0.1em]
\frac{8}{315}(x - 2)^3(-12t_2x^4 + 96t_2x^3 - 288t_2x^2 + 384t_2x + 112t_1 - 192t_2 - 7)
\end{bmatrix}.
\end{align*}
}}
Hence, $\sm_\infty(\phi)=3$.
Because $\tz_{a_1}$ has the special eigenvalues $2^0, 2^{-1},\ldots, 2^{-7}$ plus additional eigenvalues $2^{-3}, 2^{-7}$, we have $\sm_\infty(a_1)=3$ because $3=\sm_2(a_1)-1/2\le \sm_\infty(a_1)\le 3$.
Hence, by \cref{thm:main,thm:lsd}, the vector subdivision scheme with mask $a$ is a $\CH{2}$ convergent Lagrange subdivision scheme.

The second family of matrix masks $a_2\in \lrs{0}{2}{2}$ with $\fs(a_2)=[-2,2]$ is given by
\begin{align*}
&a_2(0)=
\begin{bmatrix}
(\frac{40}{3}t_1^2 - \frac{71}{24}t_1 + \frac{31}{768})t_2 + \frac{5}{2}t_1 - \frac{13}{128}
&(\frac{40}{3}t_1^2 - \frac{71}{24}t_1 + \frac{31}{768})t_2
\\[0.1em]
(-\frac{40}{3}t_1^2 + \frac{71}{24}t_1 - \frac{31}{768})t_2 - 5t_1 + \frac{71}{128} - \frac{15}{32t_2}
&(-\frac{40}{3}t_1^2 + \frac{71}{24}t_1 - \frac{31}{768})t_2 - \frac{5}{2}t_1 + \frac{29}{64}
\end{bmatrix},\\
&a_2(1)=\begin{bmatrix}
\frac{1}{16}-t_1 t_2 &-t_1 t_2\\[0.1em]
\frac{3}{16}+t_1 t_2 &\frac{1}{4}+t_1t_2
\end{bmatrix},\\
&a_2(2)=\begin{bmatrix}
(-\frac{20}{3}t_1^2 - \frac{7}{48}t_1
- \frac{1}{1536})t_2 - \frac{5}{4}t_1
- \frac{1}{256}
 &(-\frac{20}{3}t_1^2 - \frac{7}{48}t_1 - \frac{1}{1536})t_2 \\[0.1em]
(\frac{20}{3}t_1^2 + \frac{7}{48}t_1 + \frac{1}{1536})t_2 + \frac{5}{2}t_1
 + \frac{7}{256} + \frac{15}{64t_2}
&(\frac{20}{3}t_1^2 + \frac{7}{48}t_1 + \frac{1}{1536}t_2 + \frac{5}{4}t_1 +
\frac{3}{128}
\end{bmatrix}
\end{align*}
with an order $8$ matching filter $\vgu_{a_2}\in \lrs{0}{1}{2}$ satisfying
\begin{align*}
\wh{\vgu_{a_2}}(\xi)
&=\left[1+(\tfrac{1}{6}-\tfrac{16}{3} t_1-\tfrac{1}{t_2})\xi^2+(\tfrac{7}{360}-\tfrac{8}{9} t_1-\tfrac{1}{6t_2})\xi^4+
(\tfrac{31}{15120}-\tfrac{14}{135}t_1-
\tfrac{7}{360t_2})\xi^6,\right. \\
&\qquad\qquad \left. 1+(\tfrac{1}{6}-\tfrac{16}{3}t_1)\xi^2
+(\tfrac{7}{360}-\tfrac{8}{9}t_1)\xi^4+(\tfrac{31}{15120}
-\tfrac{14}{135}t_1)\xi^6\right]
+\bo(|\xi|^8),\qquad \xi \to 0,
\end{align*}
where $t_1,t_2\in \R$ with $t_2\ne 0$. By calculation, we have $\sr(a_2)=8$, $\sm_2(a_2)=11/2$, and
its unique refinable vector function $\phi=[\phi_1,\phi_2]^\tp$
has symmetry $\phi(-\cdot)=\phi$ and
is a vector spline with support $[-2,2]$ given by
{\small{\[
\phi|_{[0,1]}=
\begin{bmatrix}
\frac{t_2}{1260}\left[
(960t_1 - 45) x^7 +
(350-4480t_1) x^6
+(10080t_1 - 861)x^5
+ (770-11200t_1)x^4 \right.\\
\qquad \left. +
(8960t_1 - 280)x^2 - 5120t_1 + 64\right]\\
(-\frac{16}{21}t_1t_2 + \frac{1}{28}t_2 - \frac{1}{7})x^7 +
(\frac{32}{9}t_1t_2 - \frac{5}{18}t_2 + \frac{2}{3})x^6
-(8t_1t_2 - \frac{41}{60}t_2 + \frac{3}{2})x^5\\
\qquad\qquad \qquad + (\frac{80}{9}t_1t_2 - \frac{11}{18}t_2 + \frac{5}{3})x^4 -(\frac{64}{9}t_1t_2
 - \frac{2}{9}t_2 + \frac{4}{3})x^2 + \frac{256}{63}t_1t_2 - \frac{16}{315}t_2 + \frac{16}{21}
\end{bmatrix}
\]}}
and
{\small{\[
\phi|_{[1,2]}=
\begin{bmatrix}
\frac{t_2}{1260}(2 - x)^5\left[(320t_1+5)x^2
-(1280t_1+20)x + 160t_1 + 13\right]
\\
(x-2)^5((\frac{16}{63}t_1t_2 + \frac{1}{252}t_2
+ \frac{1}{21})x^2
-(\frac{64}{63}t_1t_2 + \frac{1}{63}t_2 + \frac{4}{21})x + \frac{8}{63}t_1t_2 + \frac{13}{1260}t_2+ \frac{1}{42}
\end{bmatrix}.
\]}}
Hence, $\sm_\infty(\phi)=5$.
Because $\tz_{a_2}$ has the special eigenvalues $2^0, 2^{-1},\ldots, 2^{-7}$ plus the additional eigenvalues $2^{-5}, 2^{-7}$, we conclude that $\sm_\infty(a_2)=5$ because $5=\sm_2(a_2)-1/2\le \sm_\infty(a_2)\le 5$.
Hence, by \cref{thm:main,thm:lsd}, the vector subdivision scheme with mask $a$ is a $\CH{4}$ convergent Lagrange subdivision scheme.

We now test the convergence rates of the vector subdivision schemes with the matrix mask $a_2\in \lrs{0}{2}{2}$ using
\begin{align*}
&u_1=\left\{\begin{bmatrix}
1\\
0\end{bmatrix}\right\}_{[0,0]},\qquad
u_2= \left\{\begin{bmatrix}\frac{31}{6}\\
-\frac{25}{6} \end{bmatrix}\right\}_{[0,0]},\quad
u_3=\left\{
\begin{bmatrix}
-\frac{1384}{315}\\
\frac{1144}{315}\end{bmatrix},
\begin{bmatrix}
\frac{301}{18}\\
-\frac{839}{63}
\end{bmatrix},
\begin{bmatrix}
-\frac{2648}{315}\\
\frac{2168}{315}\end{bmatrix},
\begin{bmatrix}
-\frac{401}{630} \\
\frac{163}{315}
\end{bmatrix}
\right\}_{[0,3]},\quad\\
&u_4=\left\{\begin{bmatrix}
-1\\
1\end{bmatrix}\right\}_{[0,0]},
\quad
u_5=\left\{
\begin{bmatrix}
-\frac{1832}{9}\\
\frac{1496}{9}
\end{bmatrix},
\begin{bmatrix}
\frac{6493}{18}\\
-\frac{5173}{18}\end{bmatrix},
\begin{bmatrix}
-\frac{1708}{9}\\
\frac{1396}{9}\end{bmatrix},
\begin{bmatrix}
-\frac{127}{18}\\
\frac{103}{18}\end{bmatrix}\right\}_{[0,3]}.
\end{align*}
Then
\begin{align*}
&\wh{\vgu_{a_2}}(\xi)\wh{u_1}(\xi)=1+\bo(|\xi|^2),\quad
\wh{\vgu_{a_2}}(\xi)\wh{u_2}(\xi)=1+\bo(|\xi|^4),\quad
\wh{\vgu_{a_2}}(\xi)\wh{u_3}(\xi)=1+\bo(|\xi|^8),\quad\\
&\wh{\vgu_{a_2}}(\xi)\wh{u_4}(\xi)=(i\xi)^2+\bo(|\xi|^4),\quad
\wh{\vgu_{a_2}}(\xi)\wh{u_5}(\xi)=(i\xi)^2+\bo(|\xi|^8),\quad \xi \to0.
\end{align*}
According to \cref{thm:fastconverg} and noting $\sm_\infty(a_2)=5$, the convergence rates are $2, 4, 5$
for approximating $\phi$ using $u_1, u_2, u_3$, respectively, while
the convergence rates are $2$ and $3$
for approximating $\phi''$ using $u_4$ and $u_5$.
See \cref{fig:ex2} for the graphs of $\phi$ and $\phi''$ and the convergence rates of the Lagrange subdivision scheme with mask $a_2$ for $(t_1,t_2)=(1,-1)$.
\end{example}

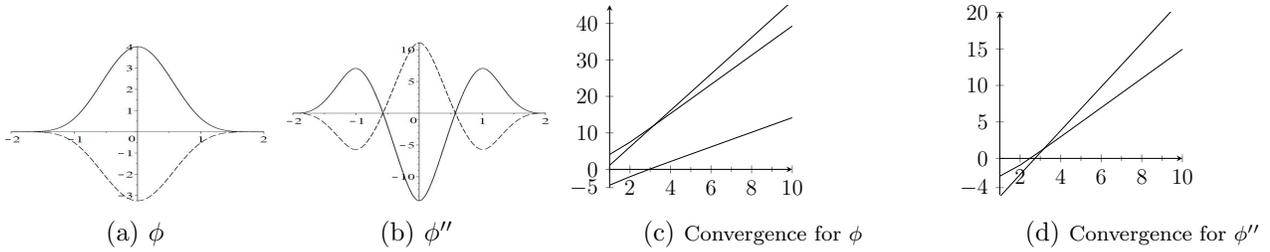
\begin{figure}[hbtp]
\begin{subfigure}[b]{0.2\textwidth}
%\raisebox{0.1cm}{
\includegraphics[width=\textwidth,height=0.6\textwidth]{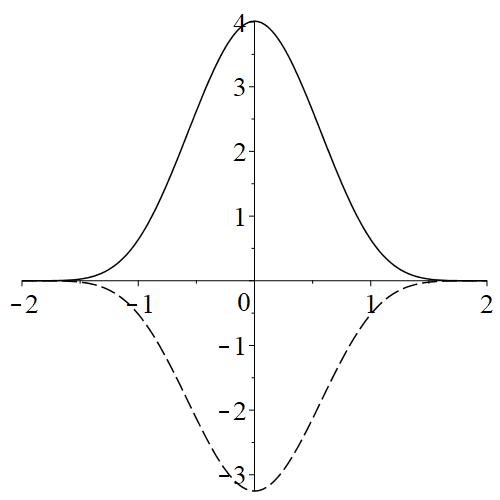}
%}
\caption{$\phi$}
\end{subfigure}
\begin{subfigure}[b]{0.2\textwidth} %\raisebox{0.1cm}{
\includegraphics[width=\textwidth,height=0.6\textwidth]{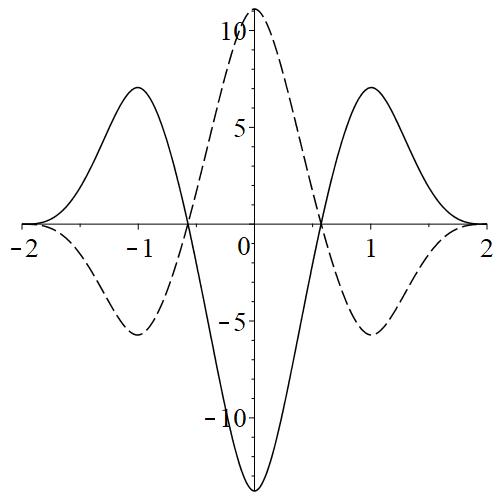}
%}
\caption{$\phi''$}
\end{subfigure}
\begin{subfigure}[b]{0.28\textwidth}
\begin{tikzpicture}[scale=0.70]
\begin{axis}
[
height=\textwidth,
width=\textwidth,
axis y line=left,
axis x line=middle,
xminorticks=true,
yminorticks=true,
minor tick num=1,
ymin=-5, ymax=45,
xmin=1, xmax=10,
xtick={0,2,4,6,8,10},
ytick={-5,0,10,20,30,40},
]
\addplot[] coordinates
{
(1,-4.32530)
(2,-2.03069)
(3,0.104175)
(4,2.14328)
(5,4.15343)
(6,6.15600)
(7,8.15665)
(8,10.1568)
(9,12.1569)
(10,14.1569)
};
\addplot[] coordinates
{
(1,4.06921)
(2,7.43446)
(3,11.3115)
(4,15.2824)
(5,19.2752)
(6,23.2733)
(7,27.2729)
(8,31.2728)
(9,35.2729)
(10,39.2727)
};
\addplot[] coordinates
{
(1,1.09346)
(2,6.11992)
(3,11.1266)
(4,16.1284)
(5,21.1288)
(6,26.1289)
(7,31.1288)
(8,36.1290)
(9,41.1289)
(10,46.1290)
};
\end{axis}
\end{tikzpicture}
\caption{{\tiny{Convergence for $\phi$}}}
\end{subfigure}
\begin{subfigure}[b]{0.28\textwidth}
\begin{tikzpicture}[scale=0.70]
\begin{axis}
[
height=\textwidth,
width=\textwidth,
axis y line=left,
axis x line=middle,
xminorticks=true,
yminorticks=true,
minor tick num=1,
ymin=-5, ymax=20,
xmin=1, xmax=10,
xtick={0,2,4,6,8,10},
ytick={-4,0,5,10,15,20},
]
\addplot[] coordinates
{
(1,-2.43557)
(2,-0.918534)
(3,0.982343)
(4,2.95860)
(5,4.95273)
(6,6.95127)
(7,8.95091)
(8,10.9508)
(9,12.9508)
(10,14.9508)
};
\addplot[] coordinates
{
(1,-5.26642)
(2,-2.21179)
(3,0.802197)
(4,3.80573)
(5,6.80662)
(6,9.80685)
(7,12.8069)
(8,15.8069)
(9,18.8069)
(10,21.8070)
};
\end{axis}
\end{tikzpicture}
\caption{{\tiny{Convergence for $\phi''$}}}
\end{subfigure}	
\caption{The refinable vector function $\phi=[\phi_1,\phi_2]^\tp$ and $\phi''$ in \cref{ex2} using mask $a_2$ with $(t_1,t_2)=(1,-1)$. Solid lines for $\phi_1$ and dashed lines for $\phi_2$.
(c) The convergence rates
%$2$, $4$ and $5$
are the slopes of $-\log_2 \err_{u_j}(n)$ for the subdivision level $n=1,\ldots,10$ with $j=1,2,3$ for computing $\phi$.
(d) The convergence rates
%$2$ and $3$
are the slopes of $-\log_2 \err_{u_j}(n)$ for  $n=1,\ldots,10$ with $j=4,5$ for computing $\phi''$.
}
\label{fig:ex2}
\end{figure}

Many Hermite subdivision schemes were reported in \cite{dm09,ms19,han20,han21} and references therein. Here we only provide an example of Hermite subdivision schemes with fast convergence rates.

\begin{example}\label{ex3}
\normalfont
Solving the linear equations in \eqref{sr:ma} with $m_a:=4$ and $\wh{\vgu_a}(\xi)=[1,i\xi]+\bo(|\xi|^{m_a})$ as $\xi\to 0$,
all symmetric masks $a\in \lrs{0}{2}{2}$ with $\fs(a)\subseteq [-2,2]$ and $\sr(a)\ge 4$ are given by
{\footnotesize{\[
a=\left\{
\begin{bmatrix}
2t_1 &-3t_2\\[0.1em]
-t_1 &t_2 \\[0.1em]
\end{bmatrix},
\begin{bmatrix}
\frac{1}{4} &\frac{3}{8}\\[0.1em]
-\frac{1}{16} &-\frac{1}{16}\\[0.1em]
\end{bmatrix},
\begin{bmatrix}
\frac{1}{2}-4t_1 &0 \\[0.1em]
0 &\frac{1}{4}+4t_2\\[0.1em]
\end{bmatrix},
\begin{bmatrix}
\frac{1}{4} &-\frac{3}{8}\\[0.1em]
\frac{1}{16} &-\frac{1}{16}\\[0.1em]
\end{bmatrix},
\begin{bmatrix}
2t_1 &3t_2\\[0.1em]
t_1 &t_2 \\[0.1em]
\end{bmatrix}
\right\}_{[-2,2]},
\]
}}
where $t_1, t_2\in \R$. The mask $a$ has order $5$ sum rules if and only if $t_1=\frac{1}{128}$ and its matching filter $\vgu_a$ is given by $\wh{\vgu_a}(\xi)=[1+\frac{1}{360}\xi^4, i\xi]+\bo(|\xi|^5)$ as $\xi \to 0$.
For $(t_1,t_2)=(\frac{1}{128},-\frac{13}{512})$, we have $\sr(a)=5$ and $\sm_2(a)\approx 4.5335$. Hence, $\sm_\infty(a)\ge \sm_2(a)-0.5\approx 4.0335$ and $\phi\in (\CH{4})^2$. Therefore, by \cref{thm:main,thm:hsd},
its vector subdivision scheme is a $\CH{4}$ convergent Hermite subdivision scheme of order $2$.
The mask $a$ has order $6$ sum rules if and only if $(t_1, t_2)=(\frac{1}{128},-\frac{7}{256})$ and the matching filter $\vgu_a$ satisfies $\wh{\vgu_a}(\xi)=[1+\frac{1}{360}\xi^4, i\xi+\frac{7}{1080}(i\xi)^5]+\bo(|\xi|^6)$ as $\xi \to 0$. For $t_1=\frac{1}{128}$ and $t_2=-\frac{7}{256}$, we have $\sr(a)=6$ and $\sm_2(a)\approx 4.3266$.

For the matrix mask $a$ with $(t_1,t_2)=(\frac{1}{128},-\frac{13}{512})$,
we now test the convergence rates of the Hermite subdivision scheme using
\[
u_1=\left\{\begin{bmatrix}
1\\
0\end{bmatrix}\right\}_{[0,0]},\quad
u_2=\left\{
\begin{bmatrix}
\frac{13}{12}\\
-\frac{1}{30}\end{bmatrix},
\begin{bmatrix}
-\frac{1}{15}\\
-\frac{1}{15}
\end{bmatrix},
\begin{bmatrix}
-\frac{1}{60}\\
0\end{bmatrix}
\right\}_{[0,2]},\quad
u_3=\left\{\begin{bmatrix}
0\\
1\end{bmatrix}\right\}_{[0,0]}.
\]
Then
$\wh{\vgu_a}(\xi)\wh{u_1}(\xi)=1+\bo(|\xi|^4)$,
$\wh{\vgu_a}(\xi)\wh{u_2}(\xi)=1+\bo(|\xi|^5)$, and
$\wh{\vgu_a}(\xi)\wh{u_3}(\xi)=i\xi+\bo(|\xi|^5)$
as $\xi \to0$.
According to \cref{thm:fastconverg}, the convergence rates are $4$ and $\sm_\infty(a)$
for approximating $\phi$ using $u_1$ and $u_2$, respectively, while
the convergence rate is $\sm_\infty(a)-1$
for approximating $\phi'$ using $u_3$.
See \cref{fig:ex3} for the graphs of $\phi$ and $\phi'$ and the convergence rates of the Hermite subdivision scheme of order $2$.
Because $\wh{\vgu_a}(\xi)=[1,i\xi]+\bo(|\xi|^{4})$ as $\xi\to 0$, the equation in \eqref{vgua:hemrite} with $m=3$ is satisfied and hence, this is also an example of a fast $\CH{4}$ convergent Hermite subdivision scheme of order $2$.
\end{example}

\begin{figure}[hbtp]
\begin{subfigure}[b]{0.2\textwidth}
%\raisebox{0.1cm}{
\includegraphics[width=\textwidth,height=0.6\textwidth]{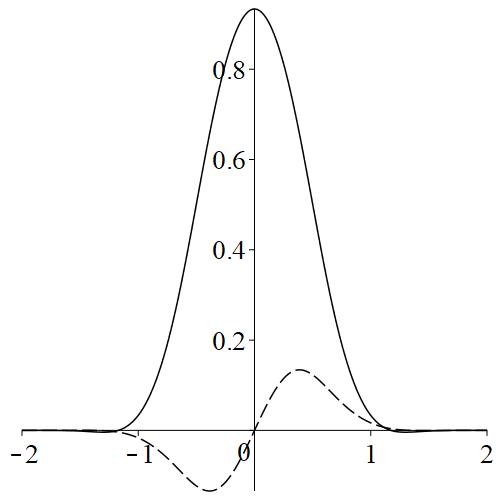}
%}
\caption{$\phi$}
\end{subfigure}
\begin{subfigure}[b]{0.2\textwidth} %\raisebox{0.1cm}{
\includegraphics[width=\textwidth,height=0.6\textwidth]{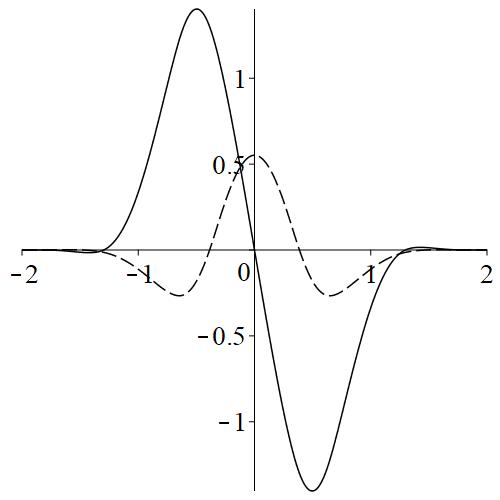}
%}
\caption{$\phi'$}
\end{subfigure}
\begin{subfigure}[b]{0.28\textwidth}
\begin{tikzpicture}[scale=0.70]
\begin{axis}
[
height=\textwidth,
width=\textwidth,
axis y line=left,
axis x line=middle,
xminorticks=true,
yminorticks=true,
minor tick num=1,
ymin=-5, ymax=45,
xmin=1, xmax=10,
xtick={0,2,4,6,8,10},
ytick={-5,0,10,20,30,40},
]
\addplot[] coordinates
{
(1,7.90692)
(2,10.7079)
(3,14.2875)
(4,17.8782)
(5,21.5834)
(6,25.3653)
(7,29.1828)
(8,33.0339)
(9,36.9082)
(10,40.8004)
%44.7073, 48.6259)
};
\addplot[] coordinates
{
(1,7.15203)
(2,10.4302)
(3,15.0116)
(4,18.9401)
(5,23.0666)
(6,27.1959)
(7,31.2871)
(8,35.4031)
(9,39.5099)
(10,43.6182)
};
\end{axis}
\end{tikzpicture}
\caption{{\tiny{Convergence for $\phi$}}}
\end{subfigure}
\begin{subfigure}[b]{0.28\textwidth}
\begin{tikzpicture}[scale=0.70]
\begin{axis}
[
height=\textwidth,
width=\textwidth,
axis y line=left,
axis x line=middle,
xminorticks=true,
yminorticks=true,
minor tick num=1,
ymin=-5, ymax=30,
xmin=1, xmax=10,
xtick={0,2,4,6,8,10},
ytick={-4,0,5,10,15,20,25,30},
]
\addplot[] coordinates
{
(1,3.36314)
(2,6.77818)
(3,10.0420)
(4,13.2365)
(5,16.3669)
(6,19.4440)
(7,22.5690)
(8,25.6726)
(9,28.7811)
(10,31.8906)
};
\end{axis}
\end{tikzpicture}
\caption{{\tiny{Convergence for $\phi'$}}}
\end{subfigure}	
\caption{The refinable vector function $\phi=[\phi_1,\phi_2]^\tp$ and $\phi'=[\phi_1',\phi_2']^\tp$ in \cref{ex3} using mask $a$ with $(t_1,t_2)=(\frac{1}{128},-\frac{13}{512})$. Solid lines for $\phi_1$ and dashed lines for $\phi_2$.
(c) The convergence rates
%(approximately $3.9$ and $4.1$)
are the slopes of $-\log_2 \err_{u_j}(n)$ for the subdivision level $n=1,\ldots,10$ with $j=1,2$ for computing $\phi$.
(d) The convergence rates
% (approximately $3.1$)
are the slopes of $-\log_2 \err_{u_3}(n)$ for  $n=1,\ldots,10$ for computing $\phi'$.
}
\label{fig:ex3}
\end{figure}

\begin{example}\label{ex4}
\normalfont
Solving the linear equations in \eqref{sr:ma} with $m_a:=4$ and $\wh{\vgu_a}(\xi)=[1,e^{i\xi/2}]+\bo(|\xi|^{m_a})$ as $\xi\to 0$,
all symmetric masks $a\in \lrs{0}{2}{2}$ with $\fs(a)\subseteq [-2,3]$ and $\sr(a)\ge 4$ are given by
{\tiny{\[
a=\left\{
\begin{bmatrix}
t_1 &-\frac{1}{32}-4t_3\\[0.1em]
0 &t_3 \\[0.1em]
\end{bmatrix},
\begin{bmatrix}
-4t_2 &\frac{9}{32}-4t_3\\[0.1em]
t_2 &-\frac{1}{32}+t_3\\[0.1em]
\end{bmatrix},
\begin{bmatrix}
\frac{1}{2}+6t_1 &\frac{9}{32}-4t_3 \\[0.1em]
-4t_1 &\frac{9}{32}+6t_3\\[0.1em]
\end{bmatrix},
\begin{bmatrix}
-4t_2 &-\frac{1}{32}-4t_3\\[0.1em]
\frac{1}{2}+6t_2 &\frac{9}{32}+6t_3\\[0.1em]
\end{bmatrix},
\begin{bmatrix}
t_1 &0\\[0.1em]
-4t_1 &-\frac{1}{32}+t_3 \\[0.1em]
\end{bmatrix},
\begin{bmatrix}
0 &0\\[0.1em]
t_2 &t_3 \\[0.1em]
\end{bmatrix}
\right\}_{[-2,3]},
\]
}}
where $t_1, t_2, t_3\in \R$.
If $(t_1, t_2,t_3)=(-\frac{1}{64},
-\frac{1}{32},-\frac{1}{128})$,
then the mask $a$ has order $5$ sum rules,
$\sm_2(a)\approx 3.8853$, and
\[
\wh{\vgu_a}(\xi)=[1, 1 + \tfrac{1}{2}(i\xi) +\tfrac{1}{8}(i\xi)^2 +\tfrac{1}{48}(i\xi)^3 + \tfrac{1}{96}(i\xi)^4]+\bo(|\xi|^5),\qquad \xi \to 0.
\]
Hence, $\sm_\infty(a)\ge \sm_2(a)-0.5\approx 3.3853$ and $\phi\in (\CH{3})^2$. Therefore, by \cref{thm:main},
the vector subdivision scheme with mask $a$ is a $\CH{3}$ convergent balanced vector subdivision scheme having order $4$ linear-phase moments for the polynomial-interpolation property (see \cite[Theorem~5.2]{han21} for details).
For the matrix mask $a$ with $(t_1, t_2,t_3)=(-\frac{1}{64},
-\frac{1}{32},-\frac{1}{128})$,
we now test the convergence rates of the vector subdivision scheme using
\[
u_1=\left\{\begin{bmatrix}0\\ 1\end{bmatrix}\right\}_{[0,0]},\quad
u_2=\left\{\begin{bmatrix}
1\\ 0\end{bmatrix}\right\}_{[0,0]},\quad
u_3=\left\{\begin{bmatrix}
-2\\ 2\end{bmatrix}\right\}_{[0,0]},\quad
u_4=\left\{\begin{bmatrix}
1\\ \frac{2}{3}\end{bmatrix},\quad
\begin{bmatrix}
\frac{1}{3}\\
-2\end{bmatrix}\right\}_{[0,1]}.
\]
Then
$\wh{\vgu_a}(\xi)\wh{u_1}(\xi)=1+\bo(|\xi|)$,
$\wh{\vgu_a}(\xi)\wh{u_2}(\xi)=1+\bo(|\xi|^5)$,
$\wh{\vgu_a}(\xi)\wh{u_3}(\xi)=i\xi+\bo(|\xi|^2)$,
and $\wh{\vgu_a}(\xi)\wh{u_4}(\xi)=i\xi+\bo(|\xi|^5)$
as $\xi \to0$.
According to \cref{thm:fastconverg}, the convergence rates are $1$ and $\sm_\infty(a)$
for approximating $\phi$ using $u_1$ and $u_2$, respectively, while
the convergence rate is $1$ and $\sm_\infty(a)-1$
for approximating $\phi'$ using $u_3$ and $u_4$.
See \cref{fig:ex4} for the graphs of $\phi$ and $\phi'$ and the convergence rates of the balanced vector subdivision scheme. Due to the balanced property (e.g., see \cite{han09}) and linear-phase moments for polynomial-interpolation property, this example of balanced vector subdivision schemes is of particular interest in multiwavelet methods for signal and image processing.
\end{example}

\begin{figure}[hbtp]
\begin{subfigure}[b]{0.2\textwidth}
%\raisebox{0.1cm}{
\includegraphics[width=\textwidth,height=0.6\textwidth]{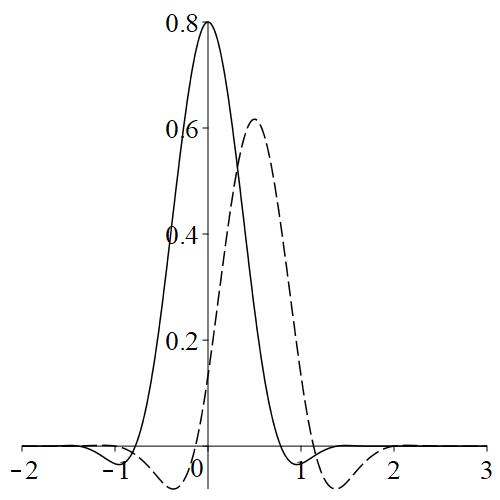}
%}
\caption{$\phi$}
\end{subfigure}
\begin{subfigure}[b]{0.2\textwidth} %\raisebox{0.1cm}{
\includegraphics[width=\textwidth,height=0.6\textwidth]{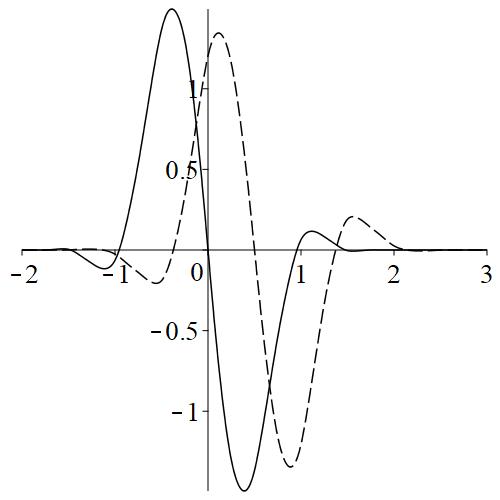}
%}
\caption{$\phi'$}
\end{subfigure}
\begin{subfigure}[b]{0.28\textwidth}
\begin{tikzpicture}[scale=0.70]
\begin{axis}
[
height=\textwidth,
width=\textwidth,
axis y line=left,
axis x line=middle,
xminorticks=true,
yminorticks=true,
minor tick num=1,
ymin=-5, ymax=45,
xmin=1, xmax=10,
xtick={0,2,4,6,8,10},
ytick={-5,0,10,20,30,40},
]
\addplot[] coordinates
{
(1,1.44338)
(2,2.45672)
(3,3.42572)
(4,4.42085)
(5,5.42070)
(6,6.42013)
(7,7.42002)
(8,8.42000)
(9,9.42001)
(10,10.4200)
};
\addplot[] coordinates
{
(1,6.32195)
(2,10.3220)
(3,14.3220)
(4,18.3220)
(5,22.3220)
(6,26.3221)
(7,30.3221)
(8,34.3220)
(9,38.3220)
(10,42.3220)
};
\end{axis}
\end{tikzpicture}
\caption{{\tiny{Convergence for $\phi$}}}
\end{subfigure}
\begin{subfigure}[b]{0.28\textwidth}
\begin{tikzpicture}[scale=0.70]
\begin{axis}
[
height=\textwidth,
width=\textwidth,
axis y line=left,
axis x line=middle,
xminorticks=true,
yminorticks=true,
minor tick num=1,
ymin=-5, ymax=30,
xmin=1, xmax=10,
xtick={0,2,4,6,8,10},
ytick={-4,0,5,10,15,20,25},
]
\addplot[] coordinates
{
(1,0.415039)
(2,1.35615)
(3,2.29956)
(4,3.27551)
(5,4.26700)
(6,5.26425)
(7,6.26341)
(8,7.26316)
(9,8.26309)
(10,9.26307)
};
\addplot[] coordinates
{
(1,3.00001)
(2,6.00002)
(3,8.86337)
(4,11.6042)
(5,14.2861)
(6,16.9112)
(7,19.4855)
(8,22.0173)
(9,24.5158)
(10,26.9893)
};
\end{axis}
\end{tikzpicture}
\caption{{\tiny{Convergence for $\phi'$}}}
\end{subfigure}	
\caption{The refinable vector function $\phi=[\phi_1,\phi_2]^\tp$ and $\phi'$ in \cref{ex4} using mask $a$ with $(t_1, t_2,t_3)=(-\frac{1}{64},
-\frac{1}{32},-\frac{1}{128})$. Solid lines for $\phi_1$ and dashed lines for $\phi_2$.
(c) The convergence rates
%(approximately $1$ and $4$)
are the slopes of $-\log_2 \err_{u_j}(n)$ for the subdivision level $1\le n\le 10$ with $j=1,2$ for computing $\phi$.
(d) The convergence rates
%(approximately $1$ and $2.4$)
are the slopes of $-\log_2 \err_{u_j}(n)$ with $j=3,4$ for $1\le n\le 10$ for computing $\phi'$.
}
\label{fig:ex4}
\end{figure}

\end{document}